\documentclass{amsart}

\usepackage[utf8]{inputenc}
\usepackage[T1]{fontenc}
\usepackage[english]{babel}
\usepackage{color}
\usepackage{amsmath} 
\usepackage{amssymb} 
\usepackage{amsthm}
\usepackage{graphicx}
\usepackage[colorlinks=true, linkcolor=blue]{hyperref}
\usepackage{cancel}

\usepackage{lmodern}
\usepackage{microtype}

\usepackage{enumerate}
\usepackage[shortlabels]{enumitem}

\theoremstyle{plain}
\newtheorem{theorem}{Theorem}[section]
\newtheorem*{theorem*}{Theorem}
\newtheorem{proposition}[theorem]{Proposition}
\newtheorem{lemma}[theorem]{Lemma}
\newtheorem{cor}[theorem]{Corollary}
\newtheorem*{cor*}{Corollary}

\newtheorem*{conjecture}{Conjecture}

\newtheorem{defi}[theorem]{Definition}

\newcommand {\R} {\mathbb{R}} \newcommand {\Z} {\mathbb{Z}}
\newcommand {\T} {\mathbb{T}} \newcommand {\N} {\mathbb{N}}

\newcommand {\p} {\partial}
\newcommand {\dt} {\partial_t}

\DeclareMathOperator{\di}{div}

\begin{document}
\title[On the inviscid Boussinesq equations]{On stability estimates for the inviscid
  Boussinesq equations}
\begin{abstract}
  We consider the (in)stability problem of the inviscid 2D Boussinesq equations
  near a combination of a shear flow $v=(y,0)$ and a stratified temperature
  $\theta=\alpha y$ with $\alpha>\frac{1}{4}$.
  We show that for any $\epsilon>0$ there exist non-trivial explicit solutions,
  which are initially perturbations of size $\epsilon$, and grow to size $1$ on a time scale
  $\epsilon^{-2}$.
  Moreover, the (simplified) linearized problem around these non-trivial states
  exhibits improved upper bounds on the possible size of norm inflation for frequencies larger and smaller
  than $\epsilon^{-4}$.
\end{abstract}
\author{Christian Zillinger}
\address{Karlsruhe Institute of Technology, Englerstraße 2,
  76131 Karlsruhe, Germany}
\email{christian.zillinger@kit.edu}
\maketitle

\section{Introduction and Main Results}
\label{sec:intro}
In this article we consider the stability of the incompressible, inviscid Boussinesq equations in a two-dimensional periodic channel
\begin{align}
  \begin{split}
  \dt v + v \cdot \nabla v + \nabla p &= \theta e_2, \\
  \dt \theta + v \cdot \nabla \theta &=0, \\
  \di (v) &= 0, \\
  (t,x,y) &\in \R^{+}\times \T \times \R, 
\end{split}
\end{align}
near the stationary solution
\begin{align}
  \label{eq:groundstate}
  \begin{split}
  v= (y,0), \  \theta= \alpha y,
\end{split}
\end{align}
where $\alpha > \frac{1}{4}$ is a constant.

The Boussinesq equations are a common model of the evolution of a heat conducting fluid in terms of its
velocity $v$ and temperature $\theta$ and may additionally incorporate viscosity
or thermal dissipation.
In particular, questions of well-posedness and asymptotic behavior in regimes
with partial dissipation \cite{tao2020stability,masmoudi2020stability,deng2020stability,wu2019stability,doering2018long,li2016global,cao2013global} or the inviscid problem \cite{elgindi2015sharp,widmayer2018convergence} have been an area with strong
research activity in recent years.

The term
\begin{align*}
  \theta e_2
\end{align*}
models buoyancy and causes hotter fluid to rise above colder fluid, where $-e_2$ is the direction of gravity.
It is well known that, in the case without shear ($v=0$), this buoyancy might
lead to the so-called Rayleigh-B\'enard instability, if hotter fluid is below
colder fluid, $\alpha<0$.

In contrast, if $\alpha>0$ is sufficiently large then the Miles-Howard criterion
\cite{howard1961note} rules out spectral instability, which is the setting
considered in this article.
As we state in Lemma \ref{lem:Etilde} and recall in Section \ref{sec:hom} the linearized equations around
the stationary solution \eqref{eq:groundstate} are stable in arbitrary Sobolev regularity globally
in time.
However, for the nonlinear equations we construct explicit solutions growing as
$\epsilon(1+t^2)^{1/4}$ as $t$ increases, which hence are only small on a time
scale $t<\epsilon^{-2}$.
Moreover, even when restricting to this time scale, higher, Gevrey regularity is required in
order to establish stability \cite{bedrossian21,tao20192d}. In the corresponding
viscous problem instead Sobolev regularity is required, however with a smallness
condition depending on the size of the viscosity \cite{zhai2022stability}.

The aims of the present article are two-fold:
\begin{itemize}
\item For related equations such as the Euler equations
  \cite{dengmasmoudi2018,dengZ2019}, Vlasov-Poisson equations \cite{Villani_long,bedrossian2016nonlinear} or partially viscous
  Boussinesq equations \cite{zillinger2021echo} it is known that the norm
  inflation of the nonlinear dynamics is tied to the interaction of non-trivial low frequency solutions, which we call traveling waves, and their
  interaction with high frequency perturbations.
  We thus construct these traveling waves for the present problem and discuss for which choices of perturbations and parameters one might expect the largest possible norm inflation.
\item For these linearized equations we identify multiple frequency regimes
  depending on the initial size $\epsilon>0$ of the waves and the time interval
  under consideration.
  For frequencies  $|\xi| < \epsilon^{-4}$ we establish an upper bound for perturbations concentrated
  at frequency $\xi$ by $\exp((\epsilon \xi)^{2/3})$. In particular, if $|\xi|\leq \epsilon^{-\alpha}$ with
  $1<\alpha\leq 4$ this factor is bounded by $\exp(\xi^{2/3(1-1/\alpha)})$.
  This bound hence matches the control by  $\exp(\sqrt{\xi})$, that is Gevrey $2$ regularity, as in the
  nonlinear problem \cite{bedrossian21} for $\xi= \epsilon^{-4}$, but exhibits
  improved bounds if $\xi$ is smaller.
  As a complementary result, if $\xi> C \epsilon^{-4}$ with a sufficiently large
r constant $C>1$, we instead obtain an upper bound which is uniform in $\xi$ and
  $\epsilon$ on the time scale under consideration.
\end{itemize}
We remark that for technical reasons we consider a simplified model, which
fixes the underlying shear flow. As we discuss in Section \ref{sec:technical} this
simplification can be removed in time intervals where the main norm inflation
takes place and for large times. For small times we provide a rough bound for
the non-simplified model, but expect that it can be improved to a uniform bound
with substantial additional technical effort.

Before stating our main results, we recall that the linearized problem around the stationary solution
\eqref{eq:groundstate} is stable, when working in coordinates moving with the
shear and choosing suitable unknowns.
The following lemma is adapted from \cite{bedrossian21,tao2020stability}.
\begin{lemma}
  \label{lem:homogeneous}
  Let $\alpha>\frac{1}{4}$. Then the linearized Boussinesq equations
  around the stationary solution \eqref{eq:groundstate} are stable in the sense
  that for any initial data $\omega, \theta$ with $\int \omega dx = \int \theta dx=0$ the energy
  \begin{align*}
    \alpha \|((\p_x^{-2}\Delta)^{-1/4}\omega)(t,x-ty,y)\|^2_{L^2} + \|((\p_x^{-2}\Delta)^{1/4}\theta)(t,x-ty,y)\|_{L^2}^2
  \end{align*}
  is bounded above and below for all times, uniformly in terms of its initial
  value, with a constant depending only on $\alpha$.
\end{lemma}
As we discuss in Section \ref{sec:homogeneous} the choice of unknowns moving
with the underlying shear flow
\begin{align}
  \label{eq:goodunknowns}
  \begin{split}
  Z(t,x,y)&:= \sqrt{\alpha} \left((\p_x^{-2}\Delta)^{-1/4}\omega\right)(t,x-ty,y), \\
  Q(t,x,y)&:= \left((\p_x^{-2}\Delta)^{1/4}\p_x\theta\right)(t,x-ty,y)
  \end{split}
\end{align}
is natural. These unknowns have previously been used in \cite{bedrossian21} and we use
the same notation. We remark that in the (partially) viscous setting
other choices of unknowns are natural \cite{adhikari2022stability,adhikari20102d,lai2021optimal,tao2020stability,tao20192d,doering2018long,cao2013global,zhai2022stability}.

We further observe that the linearized problem around the stationary solution
\eqref{eq:groundstate} in terms of $(Z,Q)$ reads
\begin{align*}
  \dt
  \begin{pmatrix}
    Z \\ Q
  \end{pmatrix}
  &=
  \begin{pmatrix}
    \frac{1}{2}\frac{\p_x(\p_y-t\p_x)}{\p_x^2+(\p_y-t\p_x)^2} & \sqrt{\alpha} \p_x (\p_x^2+(\p_y-t\p_x)^2)^{-1/2} \\
    -\sqrt{\alpha} \p_x (\p_x^2+(\p_y-t\p_x)^2)^{-1/2} & - \frac{1}{2}\frac{\p_x(\p_y-t\p_x)}{\p_x^2+(\p_y-t\p_x)^2}
  \end{pmatrix}
  \begin{pmatrix}
    Z \\Q
  \end{pmatrix}, \\
  &=: A
  \begin{pmatrix}
    Z \\Q
  \end{pmatrix}
\end{align*}
Since the operator on the right-hand-side is a (time-dependent) constant coefficient Fourier multiplier the
evolution of $(Z,Q)$ decouples in Fourier space with
respect to both $x$ and $y$. Therefore all stability estimates hold
frequency-wise and hence extend to arbitrary Sobolev, Besov or Gevrey spaces.
However, this stability can be understood as an artifact of the fact that the
stationary solution \eqref{eq:groundstate} is independent of $x$ and that
perturbations therefore decouple in frequency and cannot propagate along chains
of resonances.
For this reason, in order to capture instabilities of the nonlinear problem,
instead of a stationary solution we consider nearby $x$-dependent explicit solutions.

\begin{lemma}[compare Proposition 2.1 in \cite{zillinger2021echo} and \cite{bedrossian21}]
\label{lem:Etilde}
  Let $\alpha\geq 0$ be given, then there exist non-trivial functions $f(t)$ and $g(t)$ such that 
  \begin{align}
    \begin{split}
    \omega(t,x,y)&= -1 + f(t)\cos(x-ty), \\
    \theta(t,x,y)&= \alpha y + g(t) \sin(x-ty), \\
    v(t,x,y)&= (y,0) + \frac{1}{1+t^2}\nabla^{\perp}\cos(x-ty), 
  \end{split}
  \end{align}
  are a solution of the nonlinear inviscid Boussinesq equations for all times.
  We call these solutions \emph{traveling waves}. Moreover, if $\alpha>\frac{1}{4}$ it holds that 
  \begin{align*}
  E(t):=\frac{ |\alpha|}{\sqrt{1+t^2}}|f(t)|^2+ \sqrt{1+t^2}|g(t)|^2   
  \end{align*} 
  satisfies 
  \begin{align*}
   c E(0) \leq E(t)\leq C E(0)
  \end{align*}
  for some constants $0<c<C<\infty$ depending on $\alpha$.
\end{lemma}

We remark that in terms of the unknowns \eqref{eq:goodunknowns} these traveling waves read
\begin{align*}
  Z(t,x,y)&= \frac{f(t)}{(1+t^2)^{1/4}} \cos(x), \\
  Q(t,x,y)&= g(t)(1+t^2)^{1/4} \sin(x).
\end{align*}
They are global in time, low-frequency solutions and remain uniformly bounded in
any suitable Sobolev or Gevrey space for all times.

The functions $f(t)$ and $g(t)$ can be computed explicitly in terms of
hypergeometric functions and the above results can hence be obtained by
explicit computation (as was done in \cite{zillinger2021echo}).
Furthermore, as shown in Lemma \ref{lem:homogeneous} the stability can also be
obtained as a special case of an energy estimate similar to the one of \cite{bedrossian21}.

In the following we will consider a simplified version of the linearization of the Boussinesq equations in terms of $(Z,Q)$ 
around these traveling waves and establish upper bounds on the possible norm inflation.
The reduction of the non-simplified linearized Boussinesq equations is discussed in Section \ref{sec:technical}.

We note that for the traveling waves of Lemma \ref{lem:Etilde} for large times the
vorticity grows as
\begin{align*}
  \|\omega(t)\|_{L^{\infty}} \approx f(0) \sqrt{t}.
\end{align*}
Therefore if the traveling wave is initially of size $\epsilon>0$ it will remain
a small perturbation of the stationary state \eqref{eq:groundstate} only on time
scales
\begin{align*}
  t < \delta^2 \epsilon^{-2},
\end{align*}
where $0<\delta<0.1$ is a constant.
Hence, similarly as in \cite{bedrossian21} we restrict to studying stability and
norm inflation on that time interval.

\begin{theorem}[Stability and upper bounds on norm inflation]
\label{theorem:main}
  Let $0<\delta<0.1$ and $0<\epsilon< 0.1$ be given and
  consider the simplified linearized
  Boussinesq equations around the traveling waves
  \begin{align*}
  Z&= \frac{f(t)}{(1+t^2)^{1/4}} \cos(x), \\
  Q&= g(t)(1+t^2)^{1/4} \sin(x),
  \end{align*}
  with $f(0)=g(0)=\epsilon$.
  That is, consider the linear problem
 \begin{align*}
  \dt
  \begin{pmatrix}
    Z \\ Q
  \end{pmatrix}
  + A
  \begin{pmatrix}
    Z \\ Q
  \end{pmatrix}
  &= - 
  \begin{pmatrix}
   |\p_x|^{1/2} \Delta_t^{-1/4}(\nabla^{\perp}|\p_x|^{-1/2}\Delta_t^{-3/4} Z \cdot \nabla f(t)\cos(x))\\ 
   |\p_x|^{-1/2}\Delta_t^{1/4}(\nabla^{\perp}|\p_x|^{-1/2}\Delta_t^{-3/4} Z \cdot \nabla g(t)\cos(x))
 \end{pmatrix},\\
   \Delta_t&=\p_x^2+(\p_y-t\p_x)^2.
 \end{align*}
 Then there exists $C>0$ and $|\gamma|<\delta$ such that for any initial data
 $(Z_0,Q_0)$ whose Fourier transform satisfies
 \begin{align*}
   \sum_{k} \int \exp\left(2C \min ((\epsilon |\xi|^{1+\gamma})^{2/3-2\gamma}, \epsilon^{-2})\right) |\mathcal{F}(Z_0,Q_0)(k,\xi)|^2 d\xi \leq 1
 \end{align*}
 the corresponding solution remains regular up to a loss in the constant $C$.
 That is, for all times $t>0$ it holds that
 \begin{align}
   \label{eq:bound}
   \sum_{k} \int \exp\left(C \min((\epsilon |\xi|^{1+\gamma})^{2/3-2\gamma}, \epsilon^{-2})\right)(1+\frac{\epsilon^{-2}}{|\xi|}) |\mathcal{F}(Z,Q)(t,k,\xi)|^2 d\xi \leq 1.
  \end{align}
  Here $k \in \Z$ denotes the frequency with respect to $x$ and $\xi \in \R$
  denotes the frequency with respect to $y$.

  Moreover, there exists a constant $c=c(\alpha)$ such that if the Fourier
  transform of the initial data is supported in the region  $|\xi|\geq c
  \epsilon^{-4}$ then the stability estimate improves to a uniform estimate
  \begin{align}
    \label{eq:uniform}
    \|(Z,Q)(t)\|_{L^2} \leq 2  \|(Z_0,Q_0)(t)\|_{L^2}.
  \end{align}
\end{theorem}
Let us comment on these results:
\begin{itemize}
\item Since the traveling waves are independent of $y$, these equations decouple
  with respect to the Fourier frequency $\xi \in \R$ corresponding to $y$.
  We may hence interpret \eqref{eq:bound} as an \emph{upper bound} on the
  possible \emph{norm inflation} factor for frequency-localized initial data by
  \begin{align*}
    \exp\left(C \min ((\epsilon |\xi|^{1+\gamma})^{2/3-2\gamma}), \epsilon^{-2}\right)(1+\frac{\epsilon^{-2}}{|\xi|}).
  \end{align*}
  In particular, we emphasize that this multiplier strongly differs from the
  Euler case.
  As we discuss in Section \ref{sec:toy} we expect that this
  bound is optimal in the sense that this norm inflation is attained (possibly with
  slightly smaller constant $C$) for all frequencies $\epsilon^{-1}\leq |\xi|\leq \epsilon^{-4}$.
  However, since estimates in certain time regimes are technically very involved
  (in particular for the non-simplifed problem) in this article we only establish upper bounds.
\item A corresponding nonlinear result has been established in
  \cite{bedrossian21} using different methods with an upper bound on the norm inflation by $\exp(C
  \xi^{\sigma})$ with $\sigma>\frac{1}{2}$.
  The present result recovers this
  bound with $\sigma=\frac{1}{2}$ for $\xi=\epsilon^{-4}$ in the linearized
  problem around traveling waves.
  As major novelties, in this article we prove that for the present model:
  \begin{itemize}
  \item The upper bound on the norm inflation factor is different and, in
    particular, much smaller when $|\xi|$ is much smaller than $ \epsilon^{-4}$. 
  \item For large frequencies $|\xi|\geq c \epsilon^{-4}$ the norm inflation is
    bounded by a constant factor instead (see Proposition
    \ref{proposition:high2} for a more detailed statement).
  \end{itemize}
  Compared to the estimates of \cite{bedrossian21} we further exploit that the
  underlying traveling wave is much smaller for small times and that the time
  cut-off imposes an upper bound on the frequencies of resonances.
\item In our simplified equations we omit the term
  \begin{align*}
    \frac{1}{1+t^2}
  \begin{pmatrix}
    f(t)\Delta_t^{-1/4}(\cos(x)\p_y \Delta_t^{1/4}Z)\\
    f(t)\Delta_t^{1/4}(\cos(x)\p_y \Delta_t^{-1/4}Q)
  \end{pmatrix}
  \end{align*}
  from the linearized Boussinesq equations.
  As we discuss in Section \ref{sec:technical} in the main time regime $t>\xi^{2/3}\epsilon^{-1/3}$, where the resonance mechanism takes place, this simplification can be removed.
  For the regime of small times, for the non-simplified problem we instead obtain
  a rough growth bound by $\exp(c\sqrt{\xi})$. However, we expect that this
  bound can be improved to a uniform bound (as for the
  simplified model) with more technical effort.
\end{itemize}

The remainder of the article is structured as follows:
\begin{itemize}
\item In Section \ref{sec:hom} we discuss the linearized problem around a ground
  state \eqref{eq:groundstate} as formulated in Lemma \ref{lem:homogeneous}. In
  particular, we introduce the unknowns and system formulation used throughout
  the article.
\item In Section \ref{sec:toy} we discuss the underlying resonance mechanism for
  a toy model. In particular, this allows us to clearly present the norm
  inflation mechanism  and compare it with the Euler equations
  or the partially viscous problem. 
\item Based on the insights derived from this model we show in Sections
  \ref{sec:longtime} and \ref{sec:small} that norm inflation cannot happen
  outside a specific time interval depending on the size of $\xi$ and
  $\epsilon$. 
\item The main result of the paper is established in Section \ref{sec:echo}, where we establish bounds on the norm
  inflation achieved.
\item Finally, in Section \ref{sec:technical} we show that the previously
  derived norm inflation estimates also extend to the non-simplified model. In
  particular, the omitted terms are only non-negligible perturbations for small
  times, where they can be absorbed by a loss of Gevrey regularity. As we
  discuss, we do not expect this loss to be attained. However, since the main
  focus of this article lies in the resonance mechanism for large times, we do
  not pursue this further.
\end{itemize}

\subsection*{Notation}
\label{sec:trivial}
In this section we collect some notation used throughout the article for easier
reference.

Our main object of interest are the (simplified) linearized Boussinesq equations around the traveling waves
of Lemma \ref{lem:Etilde} which we write in the form
\begin{align*}
   \dt
   \begin{pmatrix}
     Z \\Q
   \end{pmatrix}
   + A
   \begin{pmatrix}
     Z \\Q
   \end{pmatrix}=
   R[f(t),g(t),Z,Q],
\end{align*}
where
\begin{align*}
  f(t) &\leq C \epsilon \sqrt{1+|t|}, \\
  g(t) &\leq C \epsilon (1+|t|)^{-1/2}
\end{align*}
are the coefficients of the traveling waves of Lemma \ref{lem:Etilde}.
These equations decouple after a Fourier transform in $y$. Hence we view these
equations as equations for $(\mathcal{F}_yZ)(t,x,\xi),
(\mathcal{F}_{y}Q)(t,x,\xi)$ for any fixed frequency $\xi \in \R$ and with
slight abuse of notation write $Z(t,x), Q(t,x)$ again.

These equations may equivalently be expressed as coupled system for Fourier modes $Z_{k}, Q_k$ as
stated in Definition \ref{defi:system}:
\begin{align*}
  \begin{split}
  \begin{pmatrix}
    Z_k \\ Q_k
  \end{pmatrix}(t_2)
  & = S_k(t_2,t_1)
    \begin{pmatrix}
      Z_k \\ Q_k
    \end{pmatrix}(t_1)
  \\
  & \quad  + \int_{t_1}^{t_2} S_k(t_2, t)
    \begin{pmatrix}
      c_k^{+} Z_{k+1} + c_{k}^{-} Z_{k-1} \\
      d_k^{+} Z_{k+1} + d_{k}^{-} Z_{k-1}
    \end{pmatrix} dt,
  \end{split}
\end{align*}
with the coefficient functions stated in \eqref{eq:coeff1} (for $k\pm 1\neq 0$):
\begin{align*}
  \begin{split}
  c_{k}^{\pm} &= \pm\frac{1}{2} f(t)\xi (1+(\xi/k-t)^2)^{-1/4} (1+(\xi/(k\pm 1)-t)^2)^{-3/4},\\
  d_{k}^{\pm} &= \pm \frac{1}{2} g(t)\xi \frac{k}{k\pm 1} (1+(\xi/k-t)^2)^{1/4}  (1+(\xi/(k\pm 1)-t)^2)^{-3/4}.
\end{split}
\end{align*}
Here $c_{k}^{\pm}$ is used as short-hand-notation for $c_{k}^{+}$ or
$c_{k}^{-}$. Similarly, we use $c_{k\pm 1}^{\mp}$ to refer to $c_{k+1}^{-}$ and
$c_{k-1}^{+}$.

As noted after Lemma \ref{lem:homogeneous} for simplicity of notation the
estimates of this article are stated for $L^2(dx dy)$ or $\ell^2(\Z)$ (with
respect to $k$). However, since the above system includes only nearest neighbor
interaction all estimates extend to the case of weighted $\ell^2$ spaces,
provided the weight $\lambda(k)$ is such that $|\lambda(k)/\lambda(k\pm 1)-1|$
is small enough.
In particular, this allows for $\lambda(k)= 1+ c|k|^N$  for any $N \in \N$ and
$\lambda(k)=\exp(c|k|^{s})$ for any $0<s<1$ and hence to establish stability in
Sobolev or Gevrey spaces. 

Throughout this article in several estimates it suffices to control quantities only in terms of upper and lower bounds within a constant factor.
Hence, in order to simplify notation, we sometimes approximate values. For
instance, we write
\begin{align*}
  \frac{1}{k} - \frac{1}{k+1} = \frac{1}{k(k+1)} \approx \frac{1}{k^2}
\end{align*}
to denote that for any $k \in \N, k \neq 0$ the last two terms are comparable
within a factor at most $10$.

\section{The Homogeneous Problem, Waves and Good Unknowns }
\label{sec:hom}

As remarked following Lemma \ref{lem:homogeneous} the explicit solutions of the
Boussinesq equations of the form
\begin{align*}
  \omega(t,x+ty,y) &= -1 + f(t) \cos(x), \\
  \theta(t,x+ty,y) &= \alpha y + g(t)\sin(x),
\end{align*}
may be found by inserting this ansatz into the Boussinesq equations, which reduce
to an ODE for the coefficient functions $f,g$.

In the following we provide a different perspective on these solutions as low
frequency waves.
Thus consider the perturbations
\begin{align*}
  W(t,x,y)&= \omega(t,x+ty,y)+1,\\
  F(t,x,y)&= \theta(t,x+ty,y)- \alpha y,
\end{align*}
in coordinates moving with the affine flow.
Then the full nonlinear Boussinesq equations are given by 
\begin{align}
  \label{eq:Boussinesq}
  \begin{split}
  \dt W + \nabla^{\perp}\Phi \cdot \nabla W &= \p_x F, \\
  \dt F + \nabla^{\perp}\Phi \cdot \nabla F &= -\alpha \p_x\Phi, \\
  (\p_x^2+(\p_y-t\p_x)^2)\Phi &= W,
  \end{split}
\end{align}
where we used the cancellation of $\nabla^\perp \cdot \nabla$.
In particular, the left-hand-side has a very similar structure as the Boussinesq
equations in vorticity formulation except that the equation satisfied by the
stream function perturbation $\Phi$ now is time-dependent.
With respect to these unknowns the traveling waves of Lemma \ref{lem:Etilde}
take the form
\begin{align*}
  W&=f(t)\cos(x),\\
  F&=g(t)\sin(x).
\end{align*}
They are explicit non-trivial solutions of the nonlinear problem, which are
smooth, low frequency and initially are small perturbations of $(0,0)$.

In the following subsection we discuss the linearized equation around $(0,0)$
and introduce the associated (frequency-localized) solution operators. In
particular, we show that the linearized problem around $(0,0)$ (which we call
the homogeneous problem) is stable in arbitrary regularity.
In contrast, as we sketch in Section \ref{sec:toy} for a toy model the norm inflation of the
corresponding non-linear problem is closely linked to the interaction of high
and low frequencies by means of the nonlinearity
\begin{align*}
  \nabla^{\perp}\Phi_{high} \cdot \nabla W_{low}.
\end{align*}
In particular, this mechanism is not present in the linearized problem around
$(0,0)$ but is present in the linearized problem around traveling waves.
The main aim of the remainder of this article is to show that this linearized problem around
such waves indeed captures this norm inflation mechanism and to identify the
sharp regularity classes corresponding to this norm inflation.

\subsection{Stability of the Homogeneous Problem}
\label{sec:homogeneous}

A natural first step towards understanding the nonlinear behavior of initially
small solutions of \eqref{eq:Boussinesq} is to study the linearized problem
\begin{align*}
    \begin{split}
  \dt W  &= \p_x F, \\
  \dt F  &= -\alpha \p_x\Phi, \\
  (\p_x^2+(\p_y-t\p_x)^2)\Phi &= W.
  \end{split}
\end{align*}
Given this form, we symmetrize the problem by introducing the good
unknowns \eqref{eq:goodunknowns} (as in \cite{bedrossian21}):
\begin{align*}
  Z(t,x,y) &= \sqrt{\alpha}((\p_x^{-2}\Delta)^{-1/4}\omega)(t,x+ty,y),\\ 
  Q(t,x,y) &= ((\p_x^{-2}\Delta)^{+1/4}\p_x \theta)(t,x+ty,y),\\
  \Phi(t,x,y)&= (\Delta^{-1}\omega)(t,x+ty,y). 
\end{align*}
We remark that the problem decouples after a Fourier transform in $x$ and hence
in our definition of $Z, Q$ we may choose any power of $\p_x$ instead of
$|\p_x|^{1/2}$. This particular choice is made to simplify calculations for
$\p_x^{-2}\Delta$ and to exploit slightly improved cancellation properties in
Proposition \ref{lemma:smalltimecoeff}.
For the $x$-average we omit the $|\p_x|^{1/2}$ and define
\begin{align*}
  \int Z(t,x,y) dx &= \sqrt{\alpha} |\xi|^{-1/2} \int \omega dx, \\
  \int Q(t,x,y) dx &= |\xi|^{1/2} \int \theta dx.
\end{align*}

The following proposition states the stability result of Lemma
\ref{lem:homogeneous} in terms of these unknowns.

\begin{proposition}
\label{proposition:hom}  
	Let $\alpha>\frac{1}{4}$ and consider the linear system 
	\begin{align*}
		\dt
    \begin{pmatrix}
      Z \\ Q
    \end{pmatrix} +
    \begin{pmatrix}
      -\frac{1}{2} L_t & \sqrt{\alpha}\p_x \Delta^{-1/2}_t\\
      -\sqrt{\alpha}\p_x \Delta^{-1/2}_t& \frac{1}{2} L_t
    \end{pmatrix}
                                          \begin{pmatrix}
                                            Z \\ Q
                                          \end{pmatrix} &=0\\
    \Delta_t&:= \p_x^2+(\p_y-t\p_x)^2\\
    L_t&:= \p_x(\p_y-t\p_x)  \Delta_t^{-1}.
	\end{align*}
	Then the energy 
	\begin{align*}
		E(t)= \|Z\|_{L^2}^2 + \|Q\|_{L^2}^2 + \langle Z, \frac{1}{2\sqrt{\alpha}}L_t (\p_x \Delta^{-1/2})^{-1} Q  \rangle_{L^2}
	\end{align*}
	is approximately constant in the sense that there exist constants $0<c<C<\infty$, depending only on $\alpha$, such that 
	\begin{align*}
		c E(0)\leq E(t) \leq C E(0).
	\end{align*}
\end{proposition}
We remark that all operators involved are constant coefficient Fourier
multipliers. The problem hence decouples in frequency and we may therefore
replace the $L^2$ space in the definition of $E(t)$ by any Fourier-based Hilbert
space such Sobolev spaces $H^s$, Besov spaces or Gevrey spaces and obtain the same result.
We further remark that, as a decoupled ODE system in Fourier space, the solution
operator could be computed explicitly. However, in view to later perturbed
estimates, where upper bounds are sufficient, we instead employ an energy estimate approach as for instance used in \cite{bedrossian21}.

\begin{proof}[Proof of Proposition \ref{proposition:hom}]
We observe that the operator on the right-hand-side of the equation
	\begin{align*}
		\dt \begin{pmatrix}
			 Z \\ Q
		\end{pmatrix}
		= \begin{pmatrix}
			-\frac{1}{2} L_t & \sqrt{\alpha} \p_x \Delta_t^{-1/2}\\ 
			 -\sqrt{\alpha} \p_x \Delta_t^{-1/2} & \frac{1}{2} L_t 
		\end{pmatrix}
		\begin{pmatrix}
			 Z \\Q
		\end{pmatrix}
	\end{align*}
has anti-symmetric off-diagonal entries.
It thus follows that 
	\begin{align*}
		\frac{d}{dt} (\|Z\|^2 + \|Q\|^2)/2 = -\langle L_t Z, \sqrt{\alpha} Z\rangle + \langle L_t Q, Q\rangle.  
	\end{align*}
	Since $L_t$ is a bounded operator one may already obtain a rough upper bound by employing Gronwall's lemma (we remark that at this point we do not yet require $\alpha>\frac{1}{4}$).
In order to improve this estimate we further use that also the diagonal entries are symmetric.
Therefore we may compute that
	\begin{align*}
		& \quad \frac{d}{dt}  \langle Z, \sqrt{\alpha} \frac{1}{2}L_t (\p_x \Delta_t^{-1/2})^{-1} Q  \rangle_{L^2}\\
		&=  \langle L_t Z, \sqrt{\alpha} Z\rangle - \langle L_t Q, Q\rangle \\
		& \quad + \langle Z, \dt(  \frac{1}{2}L_t (\p_x \Delta_t^{-1/2})^{-1}) Q \rangle.
	\end{align*}
Here the first two terms exactly cancel with the ones above and thus 
	\begin{align*}
		\frac{d}{dt} E =  \langle Z, \dt(  \frac{1}{2}L_t (\p_x \Delta^{-1/2})^{-1}) Q \rangle.
	\end{align*}
	Similarly as in \cite{bedrossian21,zillinger2020boussinesq,doering2018long} we observe that the operator norm of 
	\begin{align*}
		\frac{1}{2\sqrt{\alpha}}L_t (\p_x \Delta_t^{-1/2})^{-1} 
	\end{align*}
	is strictly smaller than $2$ if (and only if) $\alpha>\frac{1}{4}$.
	Therefore, in that case $E(t)$ is a positive definite bilinear form in $(Z,Q)$
  and it follows that 
	\begin{align*}
		\frac{d}{dt} E \leq C \left\| \dt(  \frac{1}{2}L_t (\p_x \Delta_t^{-1/2})^{-1})\right\| E. 
	\end{align*}
	The result hence follows by Gronwall's lemma and noting that the problem
  decouples in frequency.
	More precisely, instead of controlling the time integral of the operator norm of $\dt(  \frac{1}{2}L_t (\p_x \Delta^{-1/2})^{-1})$, 
	it suffices to control the time integral of the Fourier symbol for each fixed frequency:
	 \begin{align*}
     \frac{1}{2\sqrt{\alpha}}
     \int \left| \dt \frac{(\xi-kt)k}{k^2+(\xi-kt)^2} \frac{(k^2+(\xi-kt)^2)^{1/2}}{ik} \right| dt,
	 \end{align*}
	which is uniformly bounded.
\end{proof}
Having established stability of the linearized problem around $(Z,Q)=(0,0)$ in
the following we study the linearization around traveling waves.

\section{Echo Chains and Gevrey Regularity}
\label{sec:Gevrey5}
While the results of Section \ref{sec:homogeneous} establish linear stability of
$(Z,Q)=(0,0)$, nonlinear stability (see \cite{bedrossian21}) and asymptotic behavior
for large times are
very challenging problems.
As a first step towards understanding the (optimal) long time behavior of the
nonlinear Boussinesq equation \eqref{eq:Boussinesq} near  a shear and
hydrostatic balance, in the following we consider a toy model highlighting the role
of the nonlinearity and of traveling waves.

\subsection{A Toy Model and Optimal Gevrey Classes}
\label{sec:toy}
Resonance chains in phase-mixing problems often manifest as a
low frequency part of the solution interacting with the high frequency part by means of
the nonlinearity (see for instance \cite{dengmasmoudi2018,bedrossian2016landau,dengZ2019,zillinger2020landau}).
Based on this heuristic in this section we introduce a toy model capturing this
mechanism, which allows us to identify an expected optimal regularity class.
The main aim of the remainder of the paper then is to show that the linearized equations around a traveling wave indeed
exhibit this growth.

We recall that the nonlinear Boussinesq equations near a traveling wave read
\begin{align*}
  &\quad \dt
  \begin{pmatrix}
    Z \\ Q
  \end{pmatrix}
+
  \begin{pmatrix}
    -\frac{1}{2} L_t & \sqrt{\alpha} \p_x \Delta_t^{-1/2}\\
    -\sqrt{\alpha} \p_x \Delta_t^{-1/2} &  \frac{1}{2} L_t 
  \end{pmatrix}
                                        \begin{pmatrix}
                                          Z \\ Q
                                        \end{pmatrix}\\
 & =
  \begin{pmatrix}
    f(t) (1+t^2)^{1/4} |\p_x|^{1/2}\Delta_t^{-1/4} (\sin(x) \p_y |\p_x|^{-1/2}\Delta_t^{-3/4} Z) \\
    g(t) (1+t^2)^{-1/4}|\p_x|^{-1/2} \Delta_t^{1/4} (\sin(x) \p_y |\p_x|^{-1/2 }\Delta_t^{-3/4} Z) 
  \end{pmatrix}
  \\ &\quad +
  \begin{pmatrix}
    f(t) (1+t^2)^{-3/4} |\p_x|^{1/2} \Delta_t^{-1/4} (\sin(x) \p_y |\p_x|^{-1/2}\Delta_t^{1/4} Z) \\
    g(t) (1+t^2)^{-3/4}|\p_x|^{-1/2} \Delta_t^{1/4} (\sin(x) \p_y |\p_x|^{1/2} \Delta_t^{-1/4} Q) 
  \end{pmatrix}
  \\ & \quad +
  \begin{pmatrix}
    |\p_x|^{1/2}\Delta^{-1/4} (\nabla^\perp |\p_x|^{-1/2}\Delta_t^{-3/4}Z \cdot \nabla |\p_x|^{-1/2}\Delta_t^{1/4}Z)\\
   |\p_x|^{-1/2} \Delta^{1/4} (\nabla^\perp|\p_x|^{-1/2} \Delta_t^{-3/4}Z \cdot \nabla |\p_x|^{1/2}\Delta_t^{-1/4}Q)
  \end{pmatrix}.
\end{align*}
Here we consider $Z$ and $Q$ to be at high frequency and thus the
right-hand-side can be seen as paraproduct decomposition into low-high, high-low and
high-high frequency products.

In order to derive our toy model in a first step we ignore all terms except the low-high term,
which drives the resonance mechanism and arrive at
\begin{align*}
  \dt
  \begin{pmatrix}
    Z \\ Q
  \end{pmatrix}
  \approx 
  \begin{pmatrix}
    f(t) (1+t^2)^{1/4} |\p_x|^{1/2}\Delta_t^{-1/4} (\sin(x) \p_y |\p_x|^{-1/2}\Delta_t^{-3/4} Z) \\
    g(t) (1+t^2)^{-1/4}|\p_x|^{-1/2} \Delta_t^{1/4} (\sin(x) \p_y |\p_x|^{-1/2 }\Delta_t^{-3/4} Z) 
  \end{pmatrix}.
\end{align*}
We observe that in this model the evolution for $Z$ decouples and, since the
coefficient functions do not depend on $y$ the equations further decouples after a Fourier
transform in $y$ (which is also true for the linearized problem around a wave,
but not for the nonlinear problem).
Let thus $\xi> 1000$ be a given frequency in $y$ and suppose that $Z$ is
localized at frequency $k$ in $x$ and $\xi$ in $y$.
Then the multiplier
\begin{align*}
  \Delta^{-3/4} \leadsto (k^2)^{-3/4} (1+ (t-\frac{\xi}{k})^2)^{-3/4}
\end{align*}
is small unless $t \approx \frac{\xi}{k}$.
As in \cite{bedrossian21} for the toy model we thus consider a two-dimensional
system, which is supposed to approximate the evolution of $\mathcal{F}(Z) (k,
\xi)$ and $\mathcal{F}(Z)(k-1,\xi)$ on a time interval, where $t \approx
\frac{\xi}{k}$. More precisely, we replace powers of $\Delta_t$ by its Fourier
symbol and for simplicity of the model approximate $t-\frac{\xi}{k+1} \approx
\frac{\xi}{k}-\frac{\xi}{k+1}\approx \frac{\xi}{k^2}$ and $\frac{k}{k+1}\approx 1$. Then the toy model reads:
\begin{align*}
  \dt Z_{R} &= f(t) (1+t^2)^{1/4} \frac{\xi}{k^2} \frac{1}{(1+(t-\frac{\xi}{k})^2)^{1/4}}  ((\frac{\xi}{k^2})^2)^{-3/4 } Z_{NR}, \\
  \dt Z_{NR} &= f(t) (1+t^2)^{1/4} \frac{\xi}{k^2} ((\frac{\xi}{k^2})^2)^{-1/4} \frac{1}{(1+(t-\frac{\xi}{k})^2)^{3/4}}  Z_{R}. 
\end{align*}
In \cite{bedrossian21} the authors construct a toy model in the same way, but
further estimate $f(t)(1+t^2)^{1/4} \leq 1$ from above.
Recalling from Lemma \ref{lem:Etilde} that $f(t)\leq C \epsilon \sqrt{1+t}$ uniformly in time, this upper bound is
achieved for $t$ being comparable to $\epsilon^{-2}$.
However, for smaller times this upper bound is a (potentially large) overestimate, leading to
a larger growth bound on the chain of resonances.

More precisely, we observe that by our choice of time interval $t\approx
\frac{\xi}{k}$.
This toy model thus suggests a growth by
\begin{align*}
 \epsilon \sqrt{\frac{\xi}{k}} \sqrt{\frac{\xi}{k^2}} \leq \sqrt{\frac{\xi}{k^2}},
\end{align*}
which is potentially much smaller:
\begin{lemma}
  \label{lemma:toymodel}
  Let $\xi\geq 100$ and $k \in \N$ be given and define
  \begin{align*}
    t_0&= 2 \xi, \\
    t_{k} &= \frac{1}{2}(\frac{\xi}{k+1}+ \frac{\xi}{k}).
  \end{align*}
  Then on the time interval $I_k=(t_k, t_{k-1})$ we consider the simplified toy
  model
  \begin{align*}
    \dt Z_{R} &= 0,\\
    \dt Z_{NR} &= \epsilon (1+t^2)^{1/4} \frac{\xi}{k^2} ((\frac{\xi}{k^2})^2)^{-1/4} \frac{1}{(1+(t-\frac{\xi}{k})^2)^{3/4}}  Z_{R}. 
  \end{align*}
  Then if $Z_{NR}(t_{k})=0$ there exists a constant $1<C<10$ such that
  \begin{align*}
    Z_{NR}(t_{k-1})&= Z_{R}(t_{k}) \int  \epsilon (1+t^2)^{1/4} \frac{\xi}{k^2} ((\frac{\xi}{k^2})^2)^{-1/4} \frac{1}{(1+(t-\frac{\xi}{k})^2)^{3/4}}dt \\
    &\approx C \epsilon \sqrt{\frac{\xi}{k}} \sqrt{\frac{\xi}{k^2}} Z_{R}(t_{k}),
  \end{align*}
  where we use $\approx$ to denote upper and lower bounds within a factor $10$.
\end{lemma}

\begin{itemize}
\item We remark that here we neglected the evolution of $Z_{R}$, which in turn effects
the evolution of $Z_{NR}$. As we discuss in Section \ref{sec:echo} taking this
coupling into account results in a slightly modified growth bound by
\begin{align*}
  C \epsilon \sqrt{\frac{\xi}{k}} \left( \frac{\xi}{k^2} \right)^{\gamma}
\end{align*}
with $|\gamma-\frac{1}{2}|<\delta$ instead.
\item By the restriction on the time scale it holds that $\epsilon
  \sqrt{\frac{\xi}{k}}\approx \epsilon \sqrt{t}\leq 1$. Thus the amount of
  growth may be estimated from above by $\sqrt{\frac{\xi}{k^2}}$, uniformly in $\xi$.
  However, this is an over estimate for most values of $k$ and $\xi$.
\item As we discuss following the the proof of this lemma, the dependence on
  $\epsilon$ and $k$ here strongly differs from the one of the Euler equations.
\end{itemize}

\begin{proof}[Proof of Lemma \ref{lemma:toymodel}]
  The integral formula is immediate. We further observe that
  \begin{align*}
     \epsilon (1+t^2)^{1/4} \frac{\xi}{k^2} \left((\frac{\xi}{k^2})^2\right)^{-1/4} \approx \epsilon \left(\frac{\xi}{k}\right)^{1/2} \left(\frac{\xi}{k^2}\right)^{1/2} = \epsilon \frac{\xi}{k^{3/2}}, 
  \end{align*}
  and that
  \begin{align*}
    \int_{\R} \frac{1}{(1+(t-\frac{\xi}{k})^2)^{3/4}}dt =\frac{\sqrt{\pi}\Gamma(\frac{1}{4})}{\Gamma(\frac{3}{4})}\approx 5.2 \neq 0.
  \end{align*}
  The result hence follows by observing that if $\frac{\xi}{k^2}$ is
  sufficiently large the integral over $I_k$ is comparable to the integral over
  all of $\R$.
\end{proof}

Iterating this heuristic growth bound we may conjecture a total growth of
\begin{align}
  \sup_{k_0} \prod_{k=1}^{k_0} \frac{\epsilon \xi}{k^{3/2}} \leadsto \frac{1}{(\epsilon \xi)^{2/3}}\exp((\epsilon \xi)^{2/3})
\end{align}
by choosing $k_0\approx (\epsilon\xi)^{2/3}$ and using Stirling's
approximation.
Thus at first sight this toy model suggests stability for Gevrey $3/2$ regular
initial data, uniformly in $0<\epsilon<1$ and for all times.
However, since our evolution is restricted to the time interval
\begin{align*}
  (0, \delta \epsilon^{-2})
\end{align*}
this is an overestimate for large values of $\xi$.
Indeed, for a resonance to happen the time $t \approx
\frac{\xi}{k}$ needs to have passed.
Thus if $\xi$ is very large then in the above product we may only consider
those $k$ for which
\begin{align*}
  \frac{\xi}{k}\leq \epsilon^{-2} \Leftrightarrow k\geq \xi \epsilon^2,
\end{align*}
while by the above consideration our cascade should start at $k_0$ to maximize
the product and we thus arrive at an estimate of the possible norm inflation by
\begin{align}
  \prod_{\xi \epsilon^2\leq k \leq (\epsilon\xi)^{2/3}} \frac{\epsilon \xi}{k^{3/2}}.
\end{align}
\begin{lemma}
  \label{lemma:Stirling}
  For $0<\epsilon<0.1$ and $0<\xi< \epsilon^{-4}$ consider the function
  \begin{align*}
    G(\xi,\epsilon) := \prod_{\xi \epsilon^2\leq k \leq (\epsilon\xi)^{2/3}} \frac{\epsilon \xi}{k^{3/2}}.
  \end{align*}
  Then it holds that
  \begin{align*}
    G(\xi, \epsilon) \leq C \exp(3/2 (\epsilon \xi)^{2/3}).
  \end{align*}
  Moreover, this bound is attained in the sense that for any $1< \sigma < 4$
  we may consider $\xi=\epsilon^{-\sigma}$ and there exists $\epsilon_{\sigma}>0$
  such that for $0<\epsilon< \epsilon_{\sigma}$
  \begin{align*}
     G(\epsilon^{-\sigma}, \epsilon) \geq C \exp(0.1 (\epsilon \epsilon^{-\sigma})^{2/3})
  \end{align*}
\end{lemma}
We note that
\begin{align*}
  \xi \epsilon^2 &\leq (\epsilon\xi)^{2/3} \\
  \Leftrightarrow \xi &\leq \epsilon^{-4}
\end{align*}
and that the product is empty if $\xi$ is larger than this. In particular, for
$\xi \geq \epsilon^{-4}$ we may expect to obtain a uniform bound instead of norm
inflation (this is shown in Proposition \ref{proposition:largefreq}).

\begin{proof}[Proof of Lemma \ref{lemma:Stirling}]
  For simplicitiy of notation let $k_0=\lfloor (\epsilon\xi)^{2/3} \rfloor$.
  Then for the upper bound we may replace the starting point of the product by $1$ to obtain
  \begin{align*}
    G(\xi, \epsilon)&\leq \prod_{1\leq k \leq k_0} \frac{\epsilon \xi}{k^{3/2}}\\
    &=  \left(  (\epsilon \xi)^{2/3 k_0} /k_0! \right)^{3/2}.
  \end{align*}
  For $k_0\leq 100$ we may control this quantity by a constant uniformly in
  $\epsilon$ and $\xi$, since $(\epsilon\xi)^{2/3}\leq k_0 +1$.
  It hence suffices to discuss the case when $k_0$ is large, where by Stirling's
  approximation formula it holds that
  \begin{align*}
    k_0! \sim \sqrt{2\pi k_0} k_0^{k_0} e^{-k_0}.
  \end{align*}
  We may therefore further estimate
  \begin{align*}
     G(\xi, \epsilon)\leq e^{-3/2 k_0} (2\pi k_0)^{-3} \left( (\epsilon\xi)^{2/3}/k_0 \right)^{k_0},
  \end{align*}
  which yields the desired upper bound by noting that the last factor is
  controlled by $((k_0+1)/k_0)^{k_0}$ and hence uniformly bounded.

  For the lower bound we argue similarly, but now have to take into account that
  the product starts at $k_1:=\lfloor \epsilon^2 \xi \rfloor$ and hence
  \begin{align*}
    G(\xi, \epsilon) =  \left(  (\epsilon \xi)^{2/3 k_0} /k_0! \right)^{3/2} \left(  k_1! (\epsilon \xi)^{-2/3 k_1} \right)^{3/2}.
  \end{align*}
  We thus need to show that the second factor is not too small and hence cannot
  cancel the growth.
  Here we again may restrict to the case when $k_1$ is large and use Stirling's approximation to compute
  \begin{align*}
    k_1! (\epsilon \xi)^{-2/3 k_1} &\sim (k_1(\epsilon\xi)^{-2/3})^{k_1} e^{-k_1} \sqrt{2\pi k_1} \\
                                   &\approx (\epsilon^2 \xi (\epsilon\xi)^{-2/3})^{k_1} e^{-k_1} \sqrt{2\pi k_1} \\
                                   &= (\xi \epsilon^{4})^{k_1/3} e^{-k_1} \sqrt{2\pi k_1}\\
                                   &= \exp(-k_1 (1+ \log(\xi \epsilon^4))) \sqrt{2\pi k_1}.
  \end{align*}
  It thus suffices to estimate
  \begin{align*}
    \exp\left( 3/2 k_0 - 3/2 k_1 (1+ \log(\xi \epsilon^4)) \right).
  \end{align*}
  Here we observe that for $\xi= \epsilon^{-\sigma}$ it holds that
  \begin{align*}
    k_0 &\approx \epsilon^{(1-\sigma)\frac{2}{3}}, \\ 
    k_1 &\approx \epsilon^{2-\sigma}, \\
    \log(\xi \epsilon^4) & \approx (4-\sigma) \log(\epsilon).
  \end{align*}
  Since $\sigma>1$ the power of $\epsilon$ in the formula for $k_1$ is negative
  and for $\sigma<4$ it holds that
  \begin{align*}
    (1-\sigma)\frac{2}{3} < 2-\sigma.
  \end{align*}
  Hence for $0<\epsilon <\epsilon_{\sigma}$ sufficiently small $3/2 k_0$ dominates.
\end{proof}

Thus on this time interval the total growth is limited in terms of $\epsilon$
(by a frequency cut-off) and letting $\epsilon$ tend to zero the optimal Gevrey regularity class is expected to be
given by $2$.
\begin{conjecture}
  Let $0< \epsilon \leq 0.1$ and consider the nonlinear Boussinesq equations
  perturbed around traveling waves of size $\epsilon$.
  Then on the time interval $(0, \epsilon^{-2})$ the optimal space for stability
  is given by the Fourier weight
  \begin{align*}
    \begin{cases}
      \exp((\epsilon \xi)^{2/3}) & \text{ if } \epsilon^{-1} \leq \xi < 100 \epsilon^{-4},\\
      1 & \text{ else}.
  \end{cases}
  \end{align*}
\end{conjecture}
We remark that stability in Gevrey $2$, that is a bound by $\exp(C \sqrt{\xi})$,
has been established in \cite{bedrossian21}, which coincides with the above
weight for $\xi=\epsilon^{-4}$. The conjecture suggests an
improvement to these estimates when $\xi$ is much smaller or much larger than
$\epsilon^{-4}$ and that for $\xi < \epsilon^{-4}$ the exponent of growth
\begin{align*}
  (\epsilon\xi)^{2/3} < \xi^{1/2}
\end{align*}
is attained.
As first step towards proving this conjecture in this article we show that this
statement is true for the (simplified) linearized Boussinesq equations around traveling
waves.
We remark that the above heuristic also suggests growth bounds for
$t>\epsilon^{-2}$ for data in higher regularity classes (e.g. global in time for
Gevrey $\frac{3}{2}$).
However, at that point the toy model simplification
\begin{align*}
  \dt Z_R \approx 0
\end{align*}
ceases to be justified and the toy model has to be replaced. The question of
stability on larger time scales than $(0, \epsilon^{-2})$ for more regular data thus remains an
interesting problem for future research.

For comparison we also note that for the Euler equations and a wave initially of
size $\epsilon$ the growth of the vorticity $\omega$ (instead of $Z$) is bounded
by 
\begin{align*}
  \exp(\sqrt{\epsilon \xi})
\end{align*}
and thus stability in Gevrey $2$ regularity holds when considering arbitrarily large times \cite{dengmasmoudi2018,dengZ2019,bedrossian2013inviscid}.
When also restricting to the time interval $t<\epsilon^{-2}$ we require that
$\xi$ is such that
\begin{align*}
  \frac{\xi}{\sqrt{\epsilon \xi}} \leq \epsilon^{-2} \\
  \leadsto \xi \leq \epsilon^{-3}
\end{align*}
and thus
\begin{align*}
   \exp(\sqrt{\epsilon \xi}) \leq \exp(\min(\xi^{1/3}, \epsilon^{-1}).
\end{align*}
From this heuristic model we can thus already see that the hydrostatic balance
for $\alpha>\frac{1}{4}$ yields a strong change of the stability and norm
inflation behavior of the Boussinesq equations compared to the Euler equations:
The growth of the underlying traveling wave results in larger norm inflation.

\subsection{The Inhomogeneous Problem and Upper Bounds}
\label{sec:upperbounds}

Building on the heuristic of the previous model in the following we consider the
simplified linearized
Boussinesq equations around a traveling wave
\begin{align}
  \label{eq:Boussinesqsys}
  \dt
  \begin{pmatrix}
    Z \\Q
  \end{pmatrix}
  + A
  \begin{pmatrix}
    Z \\ Q
  \end{pmatrix}
  &=
  \begin{pmatrix}
    f(t)|\p_x|^{1/2}\Delta_t^{-1/4}(\p_y|\p_x|^{-1/2}\Delta_t^{-3/4}Z \cos(x))\\
    g(t)|\p_x|^{-1/2} \Delta_{t}^{1/4}(\p_y |\p_x|^{-1/2}\Delta_t^{-3/4}Z \cos(x))
  \end{pmatrix},
\end{align}
where we omitted the low frequency velocity contribution
\begin{align*}
  \frac{1}{1+t^2}
  \begin{pmatrix}
    f(t)|\p_x|^{1/2}\Delta_t^{-1/4}(\cos(x)\p_y |\p_x|^{-1/2}\Delta_t^{1/4}Z)\\
    f(t)|\p_x|^{-1/2}\Delta_t^{1/4}(\cos(x)\p_y |\p_x|^{-1/2}\Delta_t^{-1/4}Q)
  \end{pmatrix}.
\end{align*}
As we discuss in Section \ref{sec:technical} this term does not qualitatively
change the dynamics at large times, but is technically challenging to control
for small times.

As suggested by the notation we consider this problem as a (possibly large)
perturbation of the inhomogeneous problem of Section \ref{sec:hom}.
In particular, if we denote by
\begin{align*}
  S(t_2,t_1)
\end{align*}
the solution operator of the homogeneous problem, then we may equivalently
express the above differential equation as the integral equation
\begin{align*}
  \begin{pmatrix}
    Z\\ Q
  \end{pmatrix}
  (t_2) &= S(t_2,t_1)
  \begin{pmatrix}
    Z \\ Q
  \end{pmatrix}(t_1) \\
  & \quad + \int_{t_1}^{t_2} S(t_2,t)   \begin{pmatrix}
    f(t)|\p_x|^{1/2}\Delta_t^{-1/4}(\p_y|\p_x|^{-1/2}\Delta_t^{-3/4}Z \cos(x))\\
    g(t)|\p_x|^{1/2} \Delta_{t}^{1/4}(\p_y |\p_x|^{-1/2} \Delta_t^{-3/4}Z \cos(x))
                                               \end{pmatrix} dt
\end{align*}
We next recall that the solution operator $S(\cdot,\cdot)$ is given by a
Fourier multiplier and decouples in frequency.
Hence, taking a Fourier transform in both $x$ and $y$ we arrive at a system with
nearest neighbor interaction.
\begin{defi}[Inhomogeneous system]
  \label{defi:system}
The simplified linearized Boussinesq equations around a traveling waves for a perturbation
frequency localized at $\xi \in \R$ read
\begin{align}
  \label{eq:system}
  \begin{split}
  \begin{pmatrix}
    Z_k \\ Q_k
  \end{pmatrix}(t_2)
  & = S_k(t_2,t_1)
    \begin{pmatrix}
      Z_k \\ Q_k
    \end{pmatrix}(t_1)
  \\
  & \quad  + \int_{t_1}^{t_2} S_k(t_2, t)
    \begin{pmatrix}
      c_k^{+} Z_{k+1} + c_{k}^{-} Z_{k-1} \\
      d_k^{+} Z_{k+1} + d_{k}^{-} Z_{k-1}
    \end{pmatrix} dt
  \end{split}
\end{align}
where we introduced the coefficient functions
\begin{align}
  \label{eq:coeff1}
  \begin{split}
  c_{k}^{\pm} &= \pm\frac{1}{2} f(t)\xi (1+(\xi/k-t)^2)^{-1/4} (1+(\xi/(k\pm 1)-t)^2)^{-3/4},\\
  d_{k}^{\pm} &= \pm \frac{1}{2} g(t)\xi \frac{k}{k\pm 1} (1+(\xi/k-t)^2)^{1/4}  (1+(\xi/(k\pm 1)-t)^2)^{-3/4}.
  \end{split}
\end{align}
and denote by $Z_{k}, Q_k$ the Fourier modes at frequency $k \in \Z$ in $x$
  and frequency $\xi \in \R$ in $y$. As the system decouples in $\xi$ we treat it
  as a fixed parameter and suppress it in our notation. 
\end{defi}
The toy model of Section \ref{sec:toy} here omitted all terms except the main
resonance mechanism due to $c_{k-1}^{+}$. Our main aim in the following is to show that this model
indeed provides an accurate heuristic and that all other contributions can be
controlled.
\begin{itemize}
\item The main time regime of interest is given by time intervals $I_k$, where
  $t\approx \frac{\xi}{k}$, for which $c_{k-1}^{+}$ is comparatively large.
This regime is studied in Section \ref{sec:echo}.
  Here the coupling of modes leads to a modified growth behavior as compared to the
  toy model of Section \ref{sec:toy}.
\item In Sections \ref{sec:longtime} and \ref{sec:small} we show that in the
  remaining time intervals resonances are too small to have a large effect on
  the dynamics and the evolution is at most algebraically unstable.
\end{itemize}

\subsubsection{The Long Time Regime}
\label{sec:longtime}

In this section we consider the regime of ``large'' times, where 
\begin{align*}
  2\xi < t < \delta \epsilon^{-2}.
\end{align*}
As suggested by the heuristic model of Section \ref{sec:toy} for such large
times there are no resonances and hence the evolution is at most algebraically unstable.
\begin{proposition}
  \label{proposition:longtime}
  Let $0< \xi < 2 \delta \epsilon^{-2}$ be given and consider the time interval
  \begin{align*}
    I= (2 \xi, \delta \epsilon^{-2}).
  \end{align*}
  Then on $I$ the solution to the system \eqref{eq:system} grows at most
  algebraically in any
  Sobolev or suitable Gevrey space in the sense that the Fourier projections
  away from and onto the modes $k=-1,1$ satisfy
  \begin{align*}
    \|1_{|k|\neq 1}(Z, Q)(t)\| &\leq C_{\alpha} \exp(10) \sqrt{\frac{t}{\xi}}  \|(Z, Q)(2\xi)\|, \\
    \|1_{|k|=1}(Z, Q)(t)\| &\leq C_{\alpha, \gamma} \left(\frac{t}{\xi}\right)^{\gamma} \exp(10) \sqrt{\frac{t}{\xi}} \|(Z, Q)(2\xi)\|,
  \end{align*}
  for any $t \in I$ and any $1/2<\gamma<1$.
\end{proposition}
This proposition implies a bound on the norm inflation on this time interval by
$t^{3/2}\leq \epsilon^{-3}$, which is much smaller than the exponential growth
bound expected on earlier time intervals. We further remark that this time
interval is empty if $\xi> \epsilon^{-2}$ and that this proposition is hence
only concerned with ``small'' frequencies.

\begin{proof}[Proof of Proposition \ref{proposition:longtime}]
  We recall from Proposition \ref{proposition:hom} of Section \ref{sec:homogeneous}
  that the solution operators
  \begin{align*}
    S_k(\cdot,\cdot): \mathbb{C}^2\rightarrow \mathbb{C}^2
  \end{align*}
  are bounded by a constant $C_{\alpha}$ uniformly in $k$.
  
  It thus suffices to control the corrections of \eqref{eq:system}
  \begin{align*}
    & \quad \quad  \int_{2\xi}^{t_2} S_k(t_2, t)
    \begin{pmatrix}
      c_k^{+} Z_{k+1} + c_{k}^{-} Z_{k-1} \\
      d_k^{+} Z_{k+1} + d_{k}^{-} Z_{k-1}
    \end{pmatrix} dt
   \end{align*}
  in a suitable way to invoke Gronwall's lemma, where we will distinguish
  between the case where $|k|\geq 2$ and the cases $k=-1,0,1$.

  For simplicity of presentation in the following we establish estimates in the
  unweighted space $\ell^2$.
  The case of weighted spaces with a weight $\lambda_k$ can be reduced to this
  case by considering modified coefficient functions of the form $\frac{\lambda_{k\pm 1}}{\lambda_k}
  c_{k}^{\pm}$. More precisely, for instance for Sobolev spaces we may choose $\lambda_k= 1+
  c|k|^s$, where $c$ is a small constant and hence deduce that
  $\frac{\lambda_{k\pm 1}}{\lambda_k}$ is bounded above and below by constants
  close to $1$. Hence all estimates below extend to this case with possibly a
  small loss of constants.
  
  In the following we estimate the coefficient functions.
  We observe that for $t>2\xi$ all frequencies $k\neq 0$ are non-resonant in the
  sense that
  \begin{align*}
    |\xi-kt| \geq \frac{1}{2}t\geq |\xi|.
  \end{align*}
  In particular, recalling the definition of the coefficient functions
  \eqref{eq:coeff1} as long as none of $k, k-1, k+1$ are zero,
  we may bound
  \begin{align*}
    |c_{k}^{\pm}| + |d_{k}^{\pm}|  \leq C f(t) \frac{\xi}{t^2},
  \end{align*}
  for some universal constant $C$, where we with slight abuse of notation
  estimated $g(t)\leq f(t)/t$.
  We further recall that by our choice of time interval
  \begin{align*}
    f(t) \leq \sqrt{\delta} \ll 1
  \end{align*}
  is small and note that
  \begin{align*}
    \int_{2\xi}^\infty \frac{\xi}{t^2}= \frac{1}{2}.
  \end{align*}
  Hence, for these coefficient functions we obtain uniform $L^1$
  estimate in time (more precisely, the supremum in $k$ is still in $L^1$).

  It hence only remains to discuss the cases $k\in\{-1,0,1\}$.
  Here we observe that while
  \begin{align*}
    c_{0}^{\pm} \leq C \frac{f(t)\xi^{1/2}}{t^{3/2}}
  \end{align*}
  is uniformly integrable,
  \begin{align*}
    c_{\pm 1}^{\mp} \leq C \frac{f(t)}{\xi^{1/2} t^{1/2}}
  \end{align*}
  is not.
  
  Therefore, we cannot hope for better growth estimates than for the simple ODE
  system
  \begin{align*}
    \dt
    \begin{pmatrix}
      a \\ b
    \end{pmatrix}
    =
    \begin{pmatrix}
      0 & \frac{\xi}{t^2} \\
      \frac{1}{\sqrt{\xi}\sqrt{t}} & 0
    \end{pmatrix}
    \begin{pmatrix}
      a \\ b
    \end{pmatrix}
  \end{align*}
  for $t > \xi$. Note that after rescaling we may without loss of generality set
  $\xi=1$.
  We may then introduce $1/2<\gamma<1$  and consider
  \begin{align*}
    \dt
    \begin{pmatrix}
      a \\ (\frac{t}{\xi})^{-\gamma} b
    \end{pmatrix}
    \begin{pmatrix}
      0 & \frac{\xi^{1-\gamma}}{t^{2-\gamma}} \\
      \frac{1}{\xi^{1/2-\gamma}t^{1/2+\gamma}} & -\frac{\gamma}{t}
    \end{pmatrix}
    \begin{pmatrix}
      a \\ (\frac{t}{\xi})^{-\gamma} b
    \end{pmatrix}.
  \end{align*}
  We observe that the diagonal entries are integrable in time by our choice of
  $\gamma$, while the bottom right-entry is negative.
  Hence, by Gronwall's lemma
  \begin{align*}
    |a|^2 + |(\frac{t}{\xi})^{-\gamma} b|^2
  \end{align*}
  remains uniformly bounded, which implies that $|a|$ remains bounded while $|b|$ might
  grow algebraically.

  The claimed estimate then follows with $b= |(Z_1, Z_{-1})|$ and $a=\|(Z_k)_{k \not
  \in \{-1,1\}}\|$.
\end{proof}

\subsubsection{The Small Time or High Frequency Regime}
\label{sec:small}
By the results of the preceding Section \ref{sec:longtime} any possible norm
inflation has to happen for times 
\begin{align*}
  0< t < 2\xi.
\end{align*}
Thus similarly to the setting of the Euler or Vlasov-Poisson equations we
partition this time interval into regions in which $t$ is comparable to
$\frac{\xi}{k}$ for some $k \in \N$.
\begin{defi}
  \label{defi:Ik}
  Let $\xi >0$ be given. Then for any $k \in \N$ we define
  \begin{align*}
    t_{k} &= \frac{1}{2}(\frac{\xi}{k+1}+\frac{\xi}{k}), \\
    t_0&= 2 \xi,
  \end{align*}
  and the associated time intervals
  \begin{align*}
    I_k= (t_k, t_{k-1}).
  \end{align*}
  We further define
  \begin{align*}
    k_0=\lfloor (\epsilon |\xi|)^{2/3} \rfloor.
  \end{align*}
\end{defi}
We recall from the toy model of Section \ref{sec:toy} and from the structure of
the coefficient functions $c_{k}^{\pm},d_k^{\pm}$ stated in \eqref{eq:coeff1} that on
a given time interval $I_k$ the main resonance mechanism is expected to be
determined by
\begin{align*}
  \int_{I_k} c_{k\pm1}^{\mp} \approx \frac{\epsilon \xi}{k^{3/2}} \approx \epsilon t^{3/2} \xi^{-1/2}.
\end{align*}
In particular, this value is bigger than $1$ for $k\leq k_0$ and smaller than
$1$ if $k\geq k_0+1$.
We further remark that, if $\xi$ is much bigger than $\epsilon^{-4}$ or if $t$
is small, then we expect resonances to only result in small perturbation of the dynamics,
as we prove in the following propositions.
The main resonance mechanism in the remaining time interval is then studied in
Section \ref{sec:echo}.
As a first result we note that if $\xi>\epsilon^{-4}$ is very large all permissible choices of
$k$ (that is, with $t_k<\epsilon^{-2}$) are non-resonant and stability estimates
can be obtained by a simple ode-type estimate.
\begin{proposition}[High frequency I]
  \label{proposition:largefreq}
  Let $\xi>\epsilon^{-4}$, then there exists as constant $C$ depending only on
  $\alpha$ such that for any choice of initial data it holds that
  \begin{align*}
    \|(Z,Q)(t)\| \leq \exp(C \min(t, (\epsilon \xi)^{2/3}, \epsilon^{-2})) \|(Z,Q)(0)\|.
  \end{align*}
\end{proposition}

\begin{proof}[Proof of Proposition \ref{proposition:largefreq}]
We recall that by Definition \ref{defi:system} we may equivalently consider a
system of integral equations
\begin{align*}
    \begin{pmatrix}
    Z_k \\ Q_k
  \end{pmatrix}(t)
  & = S_k(t,0)
    \begin{pmatrix}
      Z_k \\ Q_k
    \end{pmatrix}(0)
  \\
  & \quad  + \int_{0}^{t} S_k(t, \tau)
    \begin{pmatrix}
      c_k^{+} Z_{k+1} + c_{k}^{-} Z_{k-1} \\
      d_k^{+} Z_{k+1} + d_{k}^{-} Z_{k-1}
    \end{pmatrix}(\tau) d\tau.
\end{align*}
In particular, using the bounds on $S_{k}$ of Section \ref{sec:homogeneous} it
follows that the solution satisfies the integral inequality
\begin{align*}
  \|(Z,Q)(t)\| \leq C \|(Z,Q)(0)\| + \int_0^t C \|(Z,Q)(\tau)\| \sup_{l}(|c_{l}^{\pm}(\tau)|+|d_{l}^\pm(\tau)|) d\tau,
\end{align*}
where $C$ is a constant which may depend on $\alpha$.
The claimed bound hence follows by an application of Gronwall's inequality
provided
\begin{align}
  \label{eq:simpleGronwall}
  \sup_{l, \tau \in (0,\epsilon^{-2})}(|c_{l}^{\pm}(\tau)|+|d_{l}^\pm(\tau)|) \leq 100.
\end{align}
Indeed, we observe that for $t \in I_k$ it holds that
\begin{align*}
  |c_{l}^{\pm}| \leq |f(t)|
  \begin{cases}
    \frac{\xi}{k^2} & \text{ if } l\geq k+1, \\
    \sqrt{\frac{\xi}{k^2}} & \text { if } l \in \{k-1,k+1\}, \\
    (\frac{\xi}{k^2})^{-1/2} & \text{ if } l=k, \\
    \frac{\xi}{k^2} + \xi^{-1} & \text{ if } l\leq k-2.
  \end{cases}
\end{align*}
Here we estimated
\begin{align*}
  |\frac{\xi}{k}-\frac{\xi}{l}|\geq \frac{\xi}{k^2}
\end{align*}
for $l\neq k$ and observed that since $\xi\geq \epsilon^{-4}> 2 \epsilon^{-2}$
for $k \in \{-1,0,1\}$ we estimate $|\xi-kt|\geq \frac{1}{2}\xi$.
It thus only remains to observe that
\begin{align*}
  \frac{\xi}{k^2} = \left(\frac{\xi}{k}\right)^2 \xi^{-1} \approx t^2 \xi^{-1}<1
\end{align*}
is uniformly bounded by assumption on $\xi$ and that $f(t)< 1$ by our choice of
time interval.
The estimates on $d_{l}^{\pm}$ follow analogously by noting that
$g(t)\sqrt{1+t}$ is uniformly bounded by our choice of time interval. Thus the
estimate \eqref{eq:simpleGronwall} holds, which concludes the proof.
\end{proof}
We remark that this bound is very rough and not expected to be sharp for most
choices of $\xi$. Indeed as suggested by the model of Section \ref{sec:toy} if
$\xi$ is much larger than $\epsilon^{-4}$ we obtain no norm inflation at all.

\begin{proposition}
  \label{proposition:high2}
  Let $C_{\alpha}$ denote the operator norm of semi-group of the homogeneous
  problem, that is
  \begin{align*}
    C_{\alpha}= \sup_{l\in \Z, t,s \in \R} |S_{l}(t,s)|.
  \end{align*}
  Then for all  $\xi> 2C_{\alpha}\epsilon^{-4}$ and all $t \in (0,\epsilon^{-2})$ it
  holds that
  \begin{align*}
    \|(Z,Q)(t)\| \leq C C_{\alpha} \|(Z,Q)\|.
  \end{align*}
    The evolution is uniformly bounded.
\end{proposition}
\begin{proof}[Proof of Proposition \ref{proposition:high2}]
  We claim that for this choice of $\xi$ it holds that
  \begin{align}
    \label{eq:linftyclaim}
   \sup_{l} \int_{(0,\delta \epsilon^{-2})} |c_{l}^{\pm}| dt \leq \frac{1}{4} C_{\alpha}^{-1}.
  \end{align}
  Recalling the integral equation
  \begin{align*}
    \begin{pmatrix}
      Z_{l}\\ Q_{l}
    \end{pmatrix}(t)
    = S_{l}(t,0)
    \begin{pmatrix}
      Z_{l} \\ Q_{l}
    \end{pmatrix}(0)
    + \int_{0}^t S(t,\tau) (c_{l}^{\pm}, d_{l}^{\pm})
    \begin{pmatrix}
      Z_{l\pm 1}\\Q_{l\pm 1}
    \end{pmatrix},
  \end{align*}
  we thus deduce that
  \begin{align*}
    |(Z_l, Q_l)|(t) \leq C_{\alpha} |(Z_l, Q_l)|(0) + \frac{1}{2}\sup_{(0,t)} |(Z_{l\pm 1}, Q_{l\pm 1})|.
  \end{align*}
  In particular, we may consider the supremum in $t\leq \tau$ on both sides and
  consider (suitably weighted) $\ell^2$ norms to obtain that
  \begin{align*}
   \|(Z,Q)\|_{\ell^2, t}:= \|\sup_{\tau \leq t} |(Z_{l}, Q_l)|(\tau)\|_{\ell^2}
  \end{align*}
  satisfies
  \begin{align*}
    \|(Z,Q)\|_{\ell^2, t} \leq C_{\alpha} \|(Z, Q)(0)\|_{\ell^2} + \frac{1}{2} \|(Z,Q)\|_{\ell^2, t}.
  \end{align*}
  Since the factor $\frac{1}{2}$ on the right-hand-side is smaller than $1$ we
  may subtract it from both sides and obtain that 
  \begin{align*}
     \|(Z,Q)\|_{\ell^2, t} \leq 2 C_{\alpha} \|(Z, Q)(0)\|_{\ell^2}.
  \end{align*}
  Finally we observe that 
  \begin{align*}
    \sup_{\tau\leq t} \|(Z_l,Q_l)(\tau)\|_{\ell^2} \leq \|(Z,Q)\|_{\ell^2,t}
  \end{align*}
  and that for a time independent function (such as $(Z,Q)(0)$) equality holds.

  It thus only remains to establish the estimates \eqref{eq:linftyclaim}. Indeed
  we observe that if $I_{l\pm 1} \subset (0,\delta \epsilon^{-2})$ then
  \begin{align*}
    \int_{I_{l\pm 1}} |c_{l}^{\pm}| dt \leq f(t_l) \frac{\xi}{l^{3/2}}\leq \epsilon (t_l)^{3/2} \xi^{-1/2} \leq \epsilon^{-2} \xi^{-2} \leq \frac{1}{4C_{\alpha}}  
  \end{align*}
  On the remaining interval and for all other $l$ we may estimate $f(t)\leq \delta$ and observe that
  \begin{align*}
    \int_{(0,\delta \epsilon^{-2})\setminus I_l} \xi \frac{1}{(l^2+(\xi-lt)^2)^{1/4}} \frac{1}{((l\pm 1)^2 +(\xi-(l\pm 1)t)^2)^{3/4}} \leq 10.
  \end{align*}
  and that for $l\neq 0$
  \begin{align*}
    & \quad \int_{I_l} \xi \frac{1}{(l^2+(\xi-lt)^2)^{1/4}} \frac{1}{((l\pm 1)^2 +(\xi-(l\pm 1)t)^2)^{3/4}} \\
    &\leq |\frac{\xi}{l^2}|^{-1/2} \int_{I_l} \frac{1}{(1+(\frac{\xi}{l}-t)^2)^{1/4}} \\
    &\leq 2 |\frac{\xi}{l^2}|^{-1/2}|\frac{\xi}{l^2}|^{+1/2}\leq 2.
  \end{align*}
  The case $l=0$ is estimated analogously.
\end{proof}

The same method of proof can also be applied for general $\xi$ when restricting
to suitably small times.
\begin{proposition}[The small time regime]
  \label{proposition:smalltime}
  Let $\xi<\epsilon^{-4}$ and define $T< \epsilon^{-2}$ such that
  \begin{align*}
    \epsilon T^{3/2} \xi^{-1/2} = \frac{1}{4C_\alpha}.
  \end{align*}
  Further suppose that $\delta< \frac{1}{C_{\alpha}}$ with $C_{\alpha}$ as in
  Proposition \ref{proposition:high2}.
  
  Then for all $0\leq t \leq \min(T,\delta \epsilon^{-2})$ it holds that
  \begin{align*}
    \|(Z,Q)(t)\|_{\ell^2} \leq C \|(Z,Q)\|_{\ell^2}.
  \end{align*}
\end{proposition}

\begin{proof}[Proof of Proposition \ref{proposition:smalltime}]
  We claim that for this choice of $T$ it holds that
  \begin{align*}
    \int_0^T |c_{l}^\pm| \leq \frac{1}{4C_\alpha}, \\
    \int_0^T |d_{l}^\pm| \leq \frac{1}{4C_\alpha}.
  \end{align*}
  The result then follows by the same argument as in the proof of Proposition  \ref{proposition:high2}.

  Indeed, we observe that for $k$ such that $t_{k}\leq T$ (that is for all $k$
  larger than $k_1$ with $t_{k_1}\approx T$) it holds that
  \begin{align*}
    \int_{I_{l \pm 1}} |c_{l}^\pm| \leq \epsilon \frac{\xi}{k^{3/2}} \leq \epsilon t_{k}^{3/2} \xi^{-1/2} \leq \frac{1}{4C_{\alpha}}.
  \end{align*}
  If instead $k$ is such that $t_{k} < T$ (or if we integrate over
  $(0,T)\setminus I_{l\pm1}$) then integral is not (yet) resonant and
  hence
  \begin{align*}
    \int_{0}^T |c_{l}^\pm| \leq \delta \leq \frac{1}{4C_{\alpha}}.
  \end{align*}
\end{proof}

For times larger than $T$ resonances are possibly very large and thus the
preceding argument does not work anymore, since estimates of the form 
\begin{align*}
  \|(Z,Q)\|_{\ell^2, t} \leq C + 2  \|(Z,Q)\|_{\ell^2, t},
\end{align*}
do not control the norm.
In the following Section \ref{sec:echo} we thus instead establish growth bounds
on each interval $I_k$ which mimic the growth of the toy model of Section
\ref{sec:toy} with slight changes to the exponent.

\subsubsection{Main Echo Chains}
\label{sec:echo}
In this section we consider the main norm inflation mechanism of the (simplified)
linearized Boussinesq equations as compared to the
toy model of Section \ref{sec:toy}.
Here, similarly to the Euler setting \cite{dengZ2019}, it turns out for large
frequencies the back-coupling between resonant and non-resonant modes results in
correction of the growth bounds, which has to be taken into account.
As we discuss in Section \ref{sec:technical} the following results remain valid
for the non-simplified linearized Boussinesq equations as well.

As a preliminary step we consider a more accurate toy model and establish
more accurate bounds
\begin{lemma}
  \label{lemma:odemodel}
  Let $\frac{\xi}{k^2} \geq 100$ be given, let $0 \leq f(t)< \delta$ and consider the differential inequalities
  \begin{align*}
    |\dt Z_{NR}| &\leq f(t) \sqrt{\frac{\xi}{k^2}} \frac{1}{(1+t^2)^{3/4}} |Z_{R}|,  \\
    |\dt Z_{R}| &\leq f(t) \left( \frac{\xi}{k^2}\right)^{-1/2} \frac{1}{(1+t^2)^{1/4}} |Z_{NR}|,  
  \end{align*}
  on the interval $(-\frac{\xi}{k^2}, \frac{\xi}{k^2})$.
  Then there exists a constant $0 < \gamma< 2\delta$ such that
  \begin{align*}
    |Z_{NR}(+\frac{\xi}{k^2})- Z_{NR}(-\frac{\xi}{k^2})| &\leq \|f\|_{L^\infty}  \left(\frac{\xi}{k^2}\right)^{1/2+\gamma} \left(|Z_{R}(-\frac{\xi}{k^2})| 
                                                           +\|f\|_{L^\infty}|Z_{NR}(-\frac{\xi}{k^2})|\right) ,\\
    |Z_{R}(+\frac{\xi}{k^2})- Z_{R}(-\frac{\xi}{k^2})| &\leq \|f\|_{L^\infty}\left(\frac{\xi}{k^2}\right)^{1/2+\gamma} \left(|Z_{NR}(-\frac{\xi}{k^2})|  + \|f\|_{L^\infty} |Z_{R}(-\frac{\xi}{k^2})|\right).    
  \end{align*}
  That is, we obtain an upper bound on the norm inflation by
  $\|f\|_{L^\infty}  (\frac{\xi}{k^2})^{1/2+\gamma}$.
\end{lemma}

  We remark that if $f(t)$ is replaced by a constant and if we consider a
  differential equation instead of an inequality, then the ODE system
    \begin{align*}
    \dt
    \begin{pmatrix}
      u \\v
    \end{pmatrix}
    = f
    \begin{pmatrix}
    0 & (\frac{\xi}{k^2})^{1/2}\frac{1}{(1+t^2)^{3/4}} & \\
    (\frac{\xi}{k^2})^{-1/2} \frac{1}{(1+t^2)^{1/4}} & 0  
    \end{pmatrix}
    \begin{pmatrix}
      u \\ v
    \end{pmatrix}
    \end{align*}
    can be solved explicitly by noting that $u$ solves
    \begin{align*}
    (1+t^2)^{3/4} \dt (1+t^2)^{1/4} \dt u = (1+t^2) \dt^2 u + \frac{1}{2}t \dt u = f^2 u.
    \end{align*}
    This is the defining equation of Legendre functions and the above
    estimates hence follow from the known asymptotics of these functions.

    The main aim of this lemma is thus to provide a more robust energy-based
    proof, which also extends to differential inequalities with time-dependent
    coefficients.

    \begin{proof}[Proof of Lemma \ref{lemma:odemodel}]
      We consider the following energy:
    \begin{align*}
      \mathcal{E}(t):=
      \begin{cases}
        |\frac{(1+t^2)^{1/4}}{\sqrt{\xi/k^2}} Z_{NR}|^2 + |Z_R|^2 & \text{ if } t< 0, \\
        |\frac{1}{\sqrt{\xi/k^2}} Z_{NR}|^2 + |(1+t^2)^{-1/4}Z_{R}|^2 & \text{ if } t>0.
      \end{cases}
    \end{align*}
    We note that $\mathcal{E}(t)$ is continuous and that both $(1+t^2)^{1/2}$ and $(1+t^2)^{-1/2}$ are decreasing on the
    respective time intervals.
    Hence by direct computation 
    \begin{align*}
      \dt \mathcal{E}(t) \leq f(t) (1+t^2)^{-1/2} \mathcal{E}(t).
    \end{align*}
    Integrating this inequality on $(-\frac{\xi}{k^2}, \frac{\xi}{k^2})$ we
    obtain that
    \begin{align*}
      \mathcal{E}(T) \leq \exp\left(\int_{-\frac{\xi}{k^2}}^{T} f(t) (1+t^2)^{-1/2}\right) \mathcal{E}(-\frac{\xi}{k^2}). 
    \end{align*}
    for all $T\in [-\frac{\xi}{k^2}, \frac{\xi}{k^2}]$.
    In particular, it holds that
    \begin{align*}
      \mathcal{E}(t)\leq \left(\frac{\xi}{k^2}\right)^\gamma  \mathcal{E}\left(-\frac{\xi}{k^2}\right).
    \end{align*}
    and by construction
    \begin{align*}
      \mathcal{E}(-\frac{\xi}{k^2}) \approx |Z_{NR}(-\frac{\xi}{k^2})|^2 + |Z_R(-\frac{\xi}{k^2})|^2.
    \end{align*}
    We have thus established upper bounds for general initial data.
   
    We next consider the differences compared to the initial data.
    By the fundamental theorem of calculus for any $-\frac{\xi}{k^2}\leq \tau
    \leq \frac{\xi}{k^2}$ it holds that 
    \begin{align*}
      |Z_{NR}(\tau) - Z_{NR}(-\frac{\xi}{k^2})| &\leq \int_{-\frac{\xi}{k^2}}^{\tau} f(t) (1+t^2)^{-3/4} |Z_R| \\
      & \leq \int_{-\frac{\xi}{k^2}}^{\tau} f(t) ((1+t^2)^{-1/2}1_{t<0} +   (1+t^2)^{-3/4}1_{t>0}) \sqrt{E}(t) dt \\
      & \leq \|f(t)\|_{L^\infty} (1+(\frac{\xi}{k^2})^2)^{\gamma} \sqrt{E}\left(-\frac{\xi}{k^2}\right) \\
      & \leq \|f(t)\|_{L^\infty} (1+(\frac{\xi}{k^2})^2)^{\gamma} (|Z_R(-\frac{\xi}{k^2})|+ |Z_{NR}(-\frac{\xi}{k^2})|).
    \end{align*}
    This is almost the desired bound except that we are still missing one factor
    of $\|f\|_{L^\infty}$.
    It is however already sufficient to estimate $Z_R$ as
    \begin{align*}
      |Z_R(\tau) - Z_R(-\frac{\xi}{k^2})| &\leq  \|f(t)\|_{L^\infty}  \int_{-\frac{\xi}{^2}}^\tau (1+t^2)^{-1/4} (|Z_{NR}(t)- Z_{NR}(-\frac{\xi}{k^2})|+ |Z_{NR}(\frac{\xi}{k^2})|)dt \\
      &\leq  \|f(t)\|_{L^\infty} \left( \frac{\xi}{k^2} \right)^{1/2} (|Z_{NR}(t)- Z_{NR}(-\frac{\xi}{k^2})|+ |Z_{NR}(\frac{\xi}{k^2})|). 
    \end{align*}
    Finally we may return to the bound for $Z_{NR}$ and split
    \begin{align*}
      |Z_{NR}(\tau) - Z_{NR}(-\frac{\xi}{k^2})| &\leq \int_{-\frac{\xi}{k^2}}^{\tau} f(t) (1+t^2)^{-3/4} (|Z_R(t)-Z_R(-\frac{\xi}{k^2})|+ |Z_R(\frac{\xi}{k^2})|
    \end{align*}
    and insert the just derived bound.
\end{proof}

Having established this improved model we next show that also the (simplified) linearized
Boussinesq equations exhibit this modified growth (as compared to the toy model
of Section \ref{sec:toy}; the non-simplified equations are studied in
Proposition \ref{proposition:mainresfull}).
Here in addition to the above growth bounds we have to take into account the
evolution by the homogeneous semigroup (see Section \ref{sec:homogeneous}).
We further recall that for the linearized Boussinesq system
\begin{align*}
  |f(t)|\approx \epsilon \sqrt{\frac{\xi}{k}}.
\end{align*}

\begin{proposition}
  \label{proposition:mainres}
  Let $0< \xi < \epsilon^{-4}$, $\alpha>\frac{1}{4}$ and $0<\epsilon<\delta$ and consider the linearized Boussinesq
  equations \eqref{eq:system} on the time interval
  \begin{align*}
    I_k=(t_{k}, t_{k-1})
  \end{align*}
  with $1\leq k \leq k_0$ and $t_{k}, t_{k-1}$ as in Definition
  \ref{defi:Ik}.
  Then there exists $C=C(\alpha)$ and $0<\gamma<\delta$ such that for all
  choices of data at time $t_{k}$ and
  all $t\in \overline{I_k}$ it holds that
  \begin{align*}
    \left\| (Z,Q) (t) \right\|_{\ell^2} \leq C \epsilon (\frac{\xi}{k})^{1/2} (\frac{\xi}{k^2})^\gamma \left\| (Z,Q) (t_k) \right\|_{\ell^2}.
  \end{align*}
\end{proposition}

\begin{cor}
  \label{cor:iterated}
  Under the same assumptions as in Proposition \ref{proposition:mainres} for any
  $l \leq k_0$ it holds that
  \begin{align*}
    \left\| (Z,Q) (t_l) \right\|_{\ell^2} &\leq \left\| (Z,Q) (t_{k_0}) \right\|_{\ell^2}\prod_{l\leq k \leq k_0}  C \epsilon \left(\frac{\xi}{k}\right)^{1/2} \left(\frac{\xi}{k^2}\right)^{\gamma}
  \end{align*}
  Thus the total possible norm inflation on $(t_{k_0}, \delta \epsilon^{-2})$ is
  bounded by the exponential factor stated in Theorem \ref{theorem:main}.
\end{cor}

\begin{proof}[Proof of Corollary \ref{cor:iterated}]
  The result follows by repeated application of the estimate of Proposition
  \ref{proposition:mainres}. In particular, choosing $t_{l}$ maximal we obtain
  the products discussed in Section \ref{sec:toy} with an additional correction
  in the exponent. 
\end{proof}

\begin{proof}[Proof of Proposition \ref{proposition:mainres}]
  Based on the structure of the homogeneous problem as studied in Section \ref{sec:homogeneous} we consider
  \begin{align*}
    E_{l}(t)&= |Z_{l}(t)|^2 + |Q_{l}(t)|^2 + \frac{1}{2\sqrt{\alpha}} \frac{(\xi-lt)}{\sqrt{l^2+(\xi-lt)^2}} \Re Z_{l}\overline{Q}_l.
  \end{align*}
  Then by the estimates of Proposition \ref{proposition:hom} it holds that
  \begin{align*}
    \dt E_l(t) &\leq C \dt\left( \frac{1}{2\sqrt{\alpha}} \frac{(\xi-lt)}{\sqrt{l^2+(\xi-lt)^2}}  \right) E_{l}(t)\\
               &\quad + 2\overline{Z_l} (c_{l}^{+}Z_{l+1}+c_{l}^{-}Z_{l-1}) \\
               &\quad + 2\overline{Q_l} (d_{l}^{+}Q_{l+1}+d_{l}^{-}Q_{l-1}) \\
               &\quad +\frac{1}{2\sqrt{\alpha}} \frac{(\xi-lt)}{\sqrt{l^2+(\xi-lt)^2}} \overline{Z_l} (d_{l}^{+}Q_{l+1}+d_{l}^{-}Q_{l-1}) \\
               &\quad +\frac{1}{2\sqrt{\alpha}} \frac{(\xi-lt)}{\sqrt{k^2+(\xi-lt)^2}} \overline{Q_l} (c_{l}^{+}Z_{l+1}+c_{l}^{-}Z_{l-1}).
  \end{align*}
  In order to remove the first term we introduce
  \begin{align*}
    \tilde{E}_{l}(t) = E_l(t) \exp\left( \int \left| \dt\left( \frac{1}{2\sqrt{\alpha}} \frac{(\xi-lt)}{\sqrt{l^2+(\xi-lt)^2}}  \right)\right| dt \right)
  \end{align*}
  and observe that
  \begin{align*}
    \dt \tilde{E}_l \leq C \sqrt{\tilde{E}_l} (|c_{l}^{+}|+|d_l^{+}|) \sqrt{\tilde{E}_{l+1}} + C \sqrt{\tilde{E}_l} (|c_{l}^{-}|+|d_l^{-}|) \sqrt{\tilde{E}_{l-1}}.
  \end{align*}
  Recalling that
  \begin{align*}
    |c_{l}^{\pm}| \leq |f(t)|
    \begin{cases}
      (\xi/k)^{1/2} (1+(t-\frac{\xi}{k})^2)^{-3/4} &\text{ if } c_{l}^{\pm}=c_{k\pm 1}^{\mp}, \\
      (\xi/k)^{-1/2} (1+(t-\frac{\xi}{k})^2)^{-1/4} &\text{ else}.
    \end{cases}
  \end{align*}
  we are thus in the framework of Lemma \ref{lemma:odemodel}.
  More precisely, we may define
  \begin{align*}
    Z_{NR}&= \sqrt{\tilde{E}_{k+1}^2 + \tilde{E}_{k-1}^2}, \\
    Z_R&= \sqrt{\sum_{l\not \in \{k-1, k+1\}} \tilde{E}_{l}^2 }. 
  \end{align*}
  Then by the above estimates these functions satisfy the assumptions of Lemma
  \ref{lemma:odemodel} with
  \begin{align*}
    0\leq |f(t)| \leq \epsilon \sqrt{\xi/k} \leq \delta.
  \end{align*}
  In particular, it follows that
  \begin{align*}
    \|(Z,Q)\|_{\ell^2} \approx |Z_{R}|^2 + |Z_{NR}|^2
  \end{align*}
  grows at most by a factor
  \begin{align*}
    1+ \|f(t)\|_{L^\infty}  \left(\frac{\xi}{k^2}\right)^{1/2+\gamma}.
  \end{align*}
  Since we are in the regime where the latter factor is bounded below, we may
  omit the $1$ at the cost of a constant factor, which proves the result.
\end{proof}

\section{On the Model Reduction}
\label{sec:technical}

In this section we discuss the non-simplified linearized Boussinesq equations
\begin{align}
  \label{eq:linBoussinesq}
  \begin{split}
   \dt
  \begin{pmatrix}
    Z \\Q
  \end{pmatrix}
  + A
  \begin{pmatrix}
    Z \\ Q
  \end{pmatrix}
  &=
  \begin{pmatrix}
    f(t)\Delta_t^{-1/4}(\p_y\Delta_t^{-3/4}Z \cos(x))\\
    g(t) \Delta_{t}^{1/4}(\p_y \Delta_t^{-3/4}Z \cos(x))
  \end{pmatrix}\\
  &\quad + \frac{1}{1+t^2}
  \begin{pmatrix}
    f(t)\Delta_t^{-1/4}(\cos(x)\p_y \Delta_t^{1/4}Z)\\
    f(t)\Delta_t^{1/4}(\cos(x)\p_y \Delta_t^{-1/4}Q)
  \end{pmatrix},
\end{split}
\end{align}
which we may also express in an integral system as in Definition
\eqref{eq:system} by introducing the coefficients
\begin{align}
  \label{eq:coeff2}
    \begin{split}
  g_{k}^{\pm} &= \pm \frac{f(t)\xi}{2(1+t^2)} (k^2+(\xi-kt)^2)^{1/4} ((k\pm 1)^2+(\xi-(k\pm 1)t)^2)^{-1/4},\\
  h_{k}^{\pm} &= \pm \frac{f(t)\xi}{2(1+t^2)} (k^2+(\xi-kt)^2)^{-1/4} ((k\pm 1)^2+(\xi-(k\pm 1)t)^2)^{1/4}.
    \end{split}
\end{align}
We observe that by our choice of time interval
\begin{align*}
  \frac{f(t)}{2(1+t^2)} \leq \min(\epsilon \frac{\sqrt{t}}{1+t^2}, \delta \frac{1}{1+t^2}) 
\end{align*}
is small and integrable and that
\begin{align*}
  f(t) \cos(x)\p_y
\end{align*}
is a transport operator which corresponds to a change of variables
\begin{align}
  \label{eq:variables}
  \begin{split}
  (x,y)&\mapsto (x,y-F(t)\sin(x)), \\
  F(t)&= \int_{0}^t \frac{f(t)}{1+t^2} d\tau \leq 2 \epsilon,
\end{split}
\end{align}
which is an analytic change of variables and a small perturbation of the
identity (for $\epsilon$ small).

In view of this smallness and in order to simplify the analysis of the model and
the presentation of the resonance mechanism, throughout this article we have considered the simplified linearized
Boussinesq equations which omitted these terms.

In the following we show that this simplification indeed does not change the results
of the long-time regime of Section \ref{sec:longtime} and the echo chains of
Section \ref{sec:echo}.
For the small time regime of Section \ref{sec:small} we obtain (much) rougher
bounds for the full model. We expect that with (considerable) technical effort
it should be possible to improve these bounds after incorporating an additional
change of variables (see the discussion following Proposition \ref{proposition:mainresfull}).

We begin by discussing the ``large time'' regime of Section \ref{sec:longtime}:
\begin{align*}
  2 \xi < t < \delta \epsilon^{-2}.
\end{align*}

\begin{proposition}
  \label{proposition:longtimefull}
  Let $\epsilon, \delta, f(t), \gamma(t)$ be as in Proposition
  \ref{proposition:longtime}. Then the solution of the linearized Boussinesq
  equations \eqref{eq:linBoussinesq} exhibits at most algebraic growth on the
  time interval $(2\xi, \delta \epsilon^{-2})$.
  More precisely, for all $t \in (2\xi, \delta \epsilon^{-2})$ the projections
  onto and away from the Fourier modes $k=-1,1$ satisfy:
    \begin{align*}
    \|1_{|k|\neq 1}(Z, Q)(t)\| \leq C_{\alpha} \exp(10) \sqrt{\frac{t}{\xi}}  \|(Z, Q)(2\xi)\|, \\
    \|1_{|k|=1}(Z, Q)(t)\| \leq C_{\alpha, \gamma} (\frac{t}{\xi})^{\gamma} \exp(10) \sqrt{\frac{t}{\xi}} \|(Z, Q)(2\xi)\|,
    \end{align*}
    where $1/2< \gamma< 1$ is a constant.
  \end{proposition}
  
  \begin{proof}[Proof of Proposition \ref{proposition:longtimefull}]
    We observe that for all $k \not \in {-1,0,1}$ for $t>2\xi$ the fractions
    \begin{align*}
      (k^2+(\xi-kt)^2)^{1/4} ((k\pm 1)^2+(\xi-(k\pm 1)t)^2)^{-1/4}
    \end{align*}
    are uniformly bounded.
    Hence, for these values of $k$ we may bound
    \begin{align*}
      |g_{k}^{\pm}| + |h_{k}^{\pm}| \leq \frac{\delta \xi}{1+t^2}.
    \end{align*}
    We further observe that
    \begin{align*}
      \int_{2\xi}^\infty \frac{\delta \xi}{1+t^2} dt \leq 2 \delta
    \end{align*}
    is integrable.
    The result hence follows by the same proof as for Proposition
    \ref{proposition:longtime} by noting that
    in the remaining cases
    \begin{align*}
      |g_{k}^{\pm}| + |h_{k}^{\pm}| \leq \frac{\delta \xi}{1+t^2} (t/\xi)^{1/2}.
    \end{align*}
  \end{proof}

We next turn to the resonant regime of Section \ref{sec:echo} which consists of
the time intervals $I_k=(t_{k}, t_{k-1})$ for which
\begin{align*}
  \epsilon \sqrt{\xi/k} \sqrt{\xi/k^2}
\end{align*}
is large.

\begin{proposition}[Bound on norm inflation]
  \label{proposition:mainresfull}
  Under the assumptions of Proposition \ref{proposition:mainres} also for the
  linearized Boussinesq equations the possible norm inflation is controlled in
  the sense that for all $t\in \overline{I_k}$ it holds that
  \begin{align*}
    \left\| (Z,Q) (t) \right\|_{\ell^2} \leq C \epsilon (\frac{\xi}{k})^{1/2} (\frac{\xi}{k^2})^\gamma \left\| (Z,Q) (t_k) \right\|_{\ell^2},
  \end{align*}
  where $C=C(\alpha)$ and $0<\gamma<\delta$ are constants
\end{proposition}

\begin{proof}[Proof of Proposition \ref{proposition:mainresfull}]
We consider the same energies and unknowns as in the proof of Proposition
\ref{proposition:mainres}, where in the computation of $\dt \tilde{E}_l$ we
obtain additional terms controlled by
\begin{align*}
  (|h_{l}^{\pm}|+|g_{l}^{\pm}|) \tilde{E}_{l\pm 1}.
\end{align*}
We now observe that for $l \not \in \{k-1,k, k+1\}$  it holds that
\begin{align*}
  (|h_{l}^{\pm}|+|g_{l}^{\pm}|) \leq \frac{f(t)}{1+t^2} \xi
\end{align*}
and for $t \in I_k$ we may further bound
\begin{align*}
  \frac{f(t)}{1+t^2} \xi \leq c \delta \frac{1}{1+(\xi/k)^2} \xi \leq c \delta (\xi/k^2)^{-1} \\
  \leq c \delta  (\xi/k^2)^{-1/2} (1+(t-\frac{\xi}{k})^2)^{-1/4}.
\end{align*}
For $l \in \{k-1,k,k+1\}$ we argue similarly and control
\begin{align*}
  & \quad \frac{f(t)}{1+(\xi/k)^2} \xi (1+(\xi/k-t)^2)^{-1/4} (1+(\xi/(k+1)-t)^2)^{1/4}\\
  & \leq f(t) (\xi/k^2)^{-1} (1+(\xi/k-t)^2)^{-1/4} (\xi/k^2)^{1/2} \\
  & \leq f(t) (\xi/k^2)^{-1/2} (1+(\xi/k-t)^2)^{-1/4} 
\end{align*}
Thus for all $l$ we may control
\begin{align*}
  (|h_{l}^{\pm}|+|g_{l}^{\pm}|) \tilde{E}_{l\pm 1} \leq c f(t) (\xi/k^2)^{-1/2} (1+(\xi/k-t)^2)^{-1/4} \tilde{E}_{l\pm 1}
\end{align*}
in the same way as a non-resonant contribution $c_{l}^\pm \neq c_{k\pm
  1}^{\mp}$.
The result hence follows by the same argument as in Proposition
\ref{proposition:mainres}.
\end{proof}

It hence only remains to discuss the ``small time'' regime of Section
\ref{sec:small}.
Here we observe that
\begin{align*}
  \int \frac{f(t)}{1+t^2} \xi \leq \int \epsilon \frac{\sqrt{t}}{1+t^2} \xi \leq c \epsilon \xi
\end{align*}
in general is \emph{not} small enough to employ the contraction argument of
Proposition \ref{proposition:smalltime} unless $\xi$ is smaller than
$\epsilon^{-1}$.
Indeed also for the change of variables \eqref{eq:variables} we cannot expect
good bounds in high Sobolev or Gevrey norms for frequencies larger than $\epsilon^{-1}$.
In order to obtain better bounds we thus have to take these changes into
account.

One option here is to consider the unknowns $(Z,Q)$ in the coordinates \eqref{eq:variables}. However, here $F(t)\sin(x)$ introduces further
nearest neighbor coupling in $x$, which makes terms such as $\Delta_t^{-1/4}$
very technically challenging to study.
A second option, which sidesteps this issue, is to restrict to studying
stability estimates in $\ell^2(\Z)$ (that is, with respect to Fourier modes in
$x$ for a fixed frequency in $y$), following the argument of
\cite{bedrossian21}.
Since $f(t)/(1+t^2)\cos(x)\p_y$ is an anti-symmetric operator in $L^2$, this
space allows us to exploit  cancellation and hence to estimate
\begin{align*}
  h_{l}^{\pm} - \frac{f(t)}{2(1+t^2)} \xi
\end{align*}
instead.

As a first preliminary result we consider an adaptation of the ``intermediate
time'' estimate of \cite[Section 6.3.2]{bedrossian21}, which allows for loss of
regularity in Gevrey $\sigma$ with $\sigma<2$ in the time interval where $t>\sqrt{\xi}$.

\begin{lemma}
  \label{lemma:smalltimelarger}
  Let $\xi$ and $T$ be as in Proposition \ref{proposition:smalltime} and suppose
  that $\sqrt{\xi}< T$.
  Then on the time interval $(\sqrt{\xi},T)$ the maximal possible norm inflation
  is bounded by
  \begin{align*}
    \exp( \delta c_{\sigma} \xi^{\sigma})
  \end{align*}
  for any $\sigma\geq \frac{1}{2}$.
\end{lemma}

\begin{proof}
  Arguing similarly as in the proof of Proposition \ref{proposition:smalltime}
  we consider the energy
  \begin{align*}
    \exp(\lambda(t) \xi^{\sigma}) \|(Z,Q)\|_{L^2}^2
  \end{align*}
  with $\lambda(t)$ decreasing in time and bounded below, still to be determined.
  Computing the time derivative it then suffices to show that $\lambda(t)$ can
  be chosen such that 
  \begin{align*}
    \dot{\lambda}(t) \xi^{\sigma} + |g_{l}^{\pm}|+ |h_{l}^\pm|.
  \end{align*}
  Indeed, we claim that
  \begin{align*}
    |g_{l}^{\pm}|+ |h_{l}^\pm| \leq f(t)
  \end{align*}
  and observe that since $t<T$ and $t>\sqrt{\xi}$ it holds that
  \begin{align*}
    f(t) \sqrt{\xi}{k} &\leq 1, \\
    \Leftrightarrow  f(t)\leq \xi^{1/2}/t &\leq \xi^{\sigma} t^{-1+2(\sigma-1/2)}.
  \end{align*}
  The result then follows by noting that $t^{-1+2(\sigma-1/2)}$ is integrable
  and hence
  \begin{align*}
    \lambda(\tau):= \lambda(\sqrt{\xi}) + \int_{\sqrt{\xi}}^\tau t^{-1+2(\sigma-1/2)} dt
  \end{align*}
  yields the desired result.

  It remains to prove the claim.
  For this purpose we observe that away from resonant frequencies, that is for
  $l\not \in \{k-1,k,k+1\}$ it holds that 
  \begin{align*}
    (l^2+(\xi-lt)^2)^{1/4} ((l\pm 1)^2 + (\xi-(l\pm1)t)^2)^{-1/4} \leq 2
  \end{align*}
  is bounded and hence  
  \begin{align*}
     |g_{l}^{\pm}|+ |h_{l}^\pm| \leq \frac{f(t)}{1+t^2} \xi \leq f(t),
  \end{align*}
  where we used $t>\sqrt{\xi}$ in the last step.

  For the resonant frequency we estimate
  \begin{align*}
    (l^2+(\xi-lt)^2)^{1/4} ((l\pm 1)^2 + (\xi-(l\pm1)t)^2)^{-1/4} \leq 2 (1+\frac{\xi}{k^2}).
  \end{align*}
  Then since $t>\sqrt{\xi}$ it holds that.
  \begin{align*}
    \frac{1}{1+t^2} \xi (1+\frac{\xi}{k^2}) \leq \frac{1}{1+t^2} (\xi + t^2) \\
    \leq \frac{1}{1+t^2} 2 t^2 \leq 2.
  \end{align*}
  This concludes the proof of the claim.
\end{proof}

We next need to consider the ``small time regime'' where $t<\min(\sqrt{\xi},T)$,
where we adapt the argument of Section 6.3.1 in \cite{bedrossian21} ( in their
notation we estimate a term similar to $\mathcal{T}_{N}^{p,1}$).
On that time interval the bound by $\sqrt{\xi}/t$ is not sufficient.
We thus need to exploit the $L^2$ cancellation, which involves
\begin{align*}
  (1+(\frac{\xi}{l}-t)^2)^{1/4}(1+(\frac{\xi}{l\pm 1}-t)^2)^{-1/4} -1.
\end{align*}

\begin{lemma}
\label{lemma:smalltimecoeff}
  Let $\xi$ and $T$ be as in Proposition \ref{proposition:smalltime}.
  Then for all $0< t< \min (\sqrt{\xi}, T)$ it holds that
  \begin{align*}
    & \quad \langle Z, |\frac{1}{1+t^2} (f(t)|\p_x|^{1/2}\Delta_t^{-1/4}(\cos(x)\p_y |\p_x|^{-1/2} \Delta_t^{1/4}Z \rangle_{\ell^2}\\
    &+ \langle Q, f(t) |\p_x|^{-1/2}\Delta_t^{+1/4}(\cos(x)\p_y |\p_x|^{1/2}\Delta_t^{-1/4} Q  \rangle_{L^2}
  \end{align*}
  is bounded by
  \begin{align*}
    \exp(c_{\sigma} \xi^{\sigma}) \|(Z,Q)(0)\|_{\ell^2}^2 
  \end{align*}
\end{lemma}
We remark that for this estimate we only establish stability in the unweighted
$\ell^2$ space, since the proof exploit that the shear
$f(t)\cos(x)\p_y$ is anti-symmetric on $L^2$.
\begin{proof}
  We observe that for $t\rightarrow \infty$ or $l \rightarrow
  \infty$
  \begin{align*}
    |\p_x|^{1/4} \Delta_{t}^{1/4} \rightarrow 1.
  \end{align*}
  We may hence exploit the fact that the operator $f(t)\cos(x)\p_y $ is
  anti-symmetric and thus have to estimate
  \begin{align*}
  f(t)/(1+t^2) \xi ((1+(\xi/l-t)^2)^{1/4}(1+(\xi/(l\pm 1)-t)^2)^{-1/4}-1).
  \end{align*}
  We claim that this term can be estimated from above: 
  \begin{align}
    \label{eq:claimsmall}
  f(t)/(1+t^2) \xi ((1+(\xi/l-t)^2)^{1/4}(1+(\xi/(l\pm 1)-t)^2)^{-1/4}-1) \leq  f(t)
  \end{align}
  This is sufficient to conclude since $f(t)\leq \delta$ by assumption and for
  $1+t\leq \sqrt{\xi}$ we may insert a factor
  \begin{align*}
    1 = (1+t)^{2\sigma} (1+t)^{-2\sigma} \leq \xi^{\sigma} (1+t)^{-2\sigma}
  \end{align*}
  and $(1+t)^{-2\sigma}$ is integrable since $\sigma>1/2$.

  It thus remains to prove the claim \eqref{eq:claimsmall}.
  We may rewrite the last factor in the term to be estimated as
  \begin{align*}
    & \quad (1+(\xi/l-t)^2)^{1/4}(1+(\xi/(l\pm 1)-t)^2)^{-1/4}-1\\
    &= (1+(\xi/(l\pm 1)-t)^2)^{-1/4} ((1+(\xi/l-t)^2)^{1/4}- (1+(\xi/(l\pm 1)-t)^2)^{1/4}).
  \end{align*}
  We first discuss the case when $l$ and $l \pm 1$ do not equal $k$.
  In this case by the intermediate value theorem there exists
  \begin{align*}
    \frac{\xi}{l} < \eta < \frac{\xi}{l\pm 1}
  \end{align*}
  such that
  \begin{align*}
    (1+(\xi/l-t)^2)^{1/4}- (1+(\xi/(l\pm 1)-t)^2)^{1/4} \leq (1+(\eta -t)^2)^{-3/4} \frac{\xi}{l(l\pm 1)}.
  \end{align*}
  Since both $l$ and $l\pm 1$ are non-resonant it follows that $|\eta-t|\geq
  \frac{\xi}{l(l\pm 1)}$ and $|\frac{\xi}{l}-t| \geq \frac{\xi}{l(l\pm 1)}$.
  Summarizing for this case we obtain that
  \begin{align*}
    \xi ((1+(\xi/l-t)^2)^{1/4}(1+(\xi/(l\pm 1)-t)^2)^{-1/4}-1) \leq C
  \end{align*}
  and thus obtain a bound by $f(t)/(1+t^2)\leq f(t)$.

  For the remaining resonant cases the potentially largest one is given by $l=k$.
  In that case we may estimate 
  \begin{align*}
    \xi ((1+(\xi/l-t)^2)^{1/4}(1+(\xi/(l\pm 1)-t)^2)^{-1/4}-1) \\
    \leq (\frac{\xi}{k})^2 (1+(\frac{\xi}{k}-t)^2)^{-1/4}. 
  \end{align*}
  We thus obtain a bound of the total contribution by
  \begin{align*}
    f(t) (1+(\frac{\xi}{k}-t)^2)^{-1/4}\leq f(t),
  \end{align*}
  which concludes the proof.  
\end{proof}

\subsection*{Acknowledgments}
Funded by the Deutsche Forschungsgemeinschaft (DFG, German Research Foundation) – Project-ID 258734477 – SFB 1173.

\bibliography{citations2}
\bibliographystyle{alpha}
\end{document}